\documentclass[12pt]{article}   
   
\usepackage{latexsym,amsfonts,amsmath,epsfig,tabularx,amsthm,enumitem}
\usepackage[toc]{appendix} 

\usepackage[disable]{todonotes}

\usepackage[draft=false,colorlinks,bookmarksnumbered,linkcolor=black, citecolor=black]{hyperref}

\theoremstyle{definition}    
   
\newtheorem{theorem}{Theorem}   
    
\newtheorem*{thm}{Theorem}   

\newtheorem{lemma}{Lemma}   
    
\newtheorem{prop}{Proposition}
 
\newtheorem{example}{Example} 
 
\newtheorem{rem}{Remark}

\usepackage{amssymb}

\renewcommand*{\epsilon}{\varepsilon}                                   
\newcommand*{\nach}{\rightarrow}                                        
\newcommand*{\sep}{\, | \,}                                             
\newcommand*{\N}{\mathbb{N}}                                            
\newcommand*{\R}{\mathbb{R}}                                            
\newcommand*{\B}{\mathcal{B}}                                           
\renewcommand*{\S}{\mathcal{S}}                                           
\newcommand*{\A}{\mathcal{A}}                                           
\newcommand*{\AI}{\mathfrak{A}}                                         
\newcommand*{\SI}{\mathfrak{S}}                                         
\renewcommand*{\P}{\mathcal{P}}                                         

\newcommand*{\leer}{\emptyset}                                          
\newcommand*{\0}{\mathcal{O}}                                           
\newcommand*{\id}{\mathrm{id}}                                          
\renewcommand*{\l}{\ell} 				                                          

\newcommand*{\norm}[1]{\left\| #1 \right\|}                             
\newcommand*{\abs}[1]{\left| #1 \right|}                                
\newcommand*{\floor}[1]{\left\lfloor #1 \right\rfloor}                  
\newcommand*{\ceil}[1]{\left\lceil #1 \right\rceil}                     
\newcommand*{\link}[1]{(\ref{#1})}                                      
\newcommand*{\distr}[2]{\left\langle #1, #2 \right\rangle}                         
\renewcommand{\max}[1]{ \mathop{\mathrm{max}}\left\{#1\right\} }        
\renewcommand{\min}[1]{ \mathop{\mathrm{min}}\left\{#1\right\} }
\newcommand*{\spann}[1]{\mathop{\mathrm{span}}\left\{#1\right\}}        
\renewcommand{\tilde}[1]{ \widetilde{#1} }        											
\newcommand{\BIGOP}[1]{\mathop  
 {\mathchoice 
        {\raise-0.22em\hbox{\huge $#1$}} 
        {\raise-0.05em\hbox{\Large $#1$}}{\hbox{\large $#1$}}{#1}
 }}

\newcommand{\wrt}{w.r.t.\ }
\newcommand{\ie}{i.e.\ }
   
\title{The Complexity of linear Tensor Product Problems in (Anti-) Symmetric Hilbert Spaces\footnote{This is an extended version of a same-named paper by the author which was published in the Journal of Approximation Theory~\cite{W12}. Here all the proofs, as well as some additional assertions, are explicitly included.}}   
 
\author{Markus Weimar\footnote{Mathematisches Institut, Universit\"{a}t Jena, Ernst-Abbe-Platz 2, 07743 Jena, Germany. Email: markus.weimar@uni-jena.de. Web: http://users.minet.uni-jena.de/\textasciitilde{}weimar.}}
\begin{document}   
   
\maketitle   
   
\begin{abstract}   
We study linear problems $S_d$ defined on tensor products of Hilbert spaces
with an additional (anti-) symmetry property. 
We construct a linear algorithm that uses finitely many continuous linear functionals
and show an explicit formula for its worst case error in terms of the eigenvalues 
$\lambda$ of the operator $W_1=S_1^\dagger S_1$ of the univariate problem.
Moreover, we show that this algorithm is optimal with respect to a wide class of
algorithms and investigate its complexity.
We clarify the influence of different (anti-) symmetry conditions
on the complexity, compared to the classical unrestricted problem.
In particular, for symmetric problems with $\lambda_1\leq1$ we give characterizations 
for polynomial tractability and strong polynomial
tractability in terms of $\lambda$ and the amount of the assumed symmetry.
Finally, we apply our results to the approximation problem of
solutions of the electronic Schr\"{o}dinger equation.
\end{abstract}       
   
{\bf Keywords:} Antisymmetry, Hilbert spaces, Tensor Products, Complexity.


\section{Introduction}\label{sect_Intro}


In the theory of linear operators $S_d\colon H_d \nach G_d$ 
defined between Hilbert spaces
it is well-known that we often observe the the so-called 
\textit{curse of dimensionality} if we deal with $d$-fold 
tensor product problems. That is, the complexity of approximating 
the operator $S_d$ by algorithms using finitely many pieces of information 
increases exponentially fast with the dimension $d$.

In the last years there have been various approaches to break this
exponential dependence on the dimension, e.g., we can
relax the error definitions. 
Another way to overcome the curse is to introduce weights in order to
shrink the space of problem elements $H_d$. In the case of 
function spaces this approach is motivated by 
the assumption that we have some additional a priori knowledge about the
importance of several (groups of) variables.

In the present paper we describe an essentially 
new kind of a priori knowledge.
We assume the problem elements $f\in H_d$ to be 
\textit{(anti-) symmetric}.
This allows us to vanquish the curse and obtain different types of tractability. 

The problem of approximating \textit{wave functions}, 
e.g., solutions of the \textit{electronic Schr\"{o}dinger equation}, 
serves as an important example from computational chemistry and physics.
In quantum physics wave functions $\Psi$ describe quantum states 
of certain $d$-particle systems. 
Formally, these functions depend on $d$ blocks of variables $y_j$, 
which represent the spacial coordinates and certain additional intrinsic parameters, 
e.g., the \textit{spin}, of each particle within the system.
Due to the \textit{Pauli principle}, the only wave functions $\Psi$ 
which are physically admissible are antisymmetric 
in the sense that 
$\Psi(y) = (-1)^{\abs{\pi}} \Psi(\pi(y))$ for all $y$ 
and all permutations $\pi$ on a subset $I\subset\{1,\ldots,d\}$
of particles with the same spin.
Here $(-1)^{\abs{\pi}}$ denotes the sign of $\pi$.
The above relation means that $\Psi$ only changes its sign if we replace 
particles by each other which possess the same spin.
For further details on this topic we refer to
\autoref{sect_wave} of this paper and the references given there.
Inspired by this application we illustrate our results with 
some simple toy examples at the end of this section.

To this end, let $H_1$ and $G_1$ be infinite dimensional separable Hilbert spaces 
of univariate functions $f\colon D\subset\R \nach \R$ 
and consider a compact linear operator 
$S_1\colon H_1\nach G_1$ with singular values $\sigma=(\sigma_j)_{j\in\N}$.
Further, let $\lambda=(\lambda_j)_{j\in\N}=(\sigma_j^2)_{j\in\N}$ denote the 
sequence of the squares of the singular values of $S_1$. 
Finally, assume $S_d\colon H_d \nach G_d$ to be the $d$-fold tensor product problem.
We want to approximate $S_d$ by linear algorithms using a finite number 
of continuous linear functionals.

By $n^{\rm ent}(\epsilon,d)$ we denote the minimal number of information operations 
needed to achieve an approximation with worst case error 
at most $\epsilon>0$ on the unit ball of $H_d$. 
The integer $n^{\rm ent}(\epsilon,d)$ is called 
\textit{information complexity} of the \textit{entire} tensor product problem.
Further, consider the subspace of all $f\in H_d$ that are \textit{fully symmetric},
i.e.,
\begin{gather*}
			f(x)=f(\pi(x)) \quad \text{for all } x\in D^d \text{ and all permutations } \pi \text{ of } \{1,\ldots,d\}.
\end{gather*}
The minimal number of linear functionals needed 
to achieve an $\epsilon$-approximation
for this subspace is denoted by $n^{\rm sym}(\epsilon,d)$.
Finally, define the subspace of all functions $f\in H_d$ 
that are \textit{fully antisymmetric} by the condition
\begin{gather*}
			f(x)=(-1)^{\abs{\pi}} f(\pi(x)) \quad \text{for all } x\in D^d \text{ and all } \pi
\end{gather*}
and denote the information complexity with respect to this subspace 
by $n^{\rm asy}(\epsilon,d)$.

Since $H_d$ is a Hilbert space, the optimal algorithm 
for the entire tensor product problem is well-known. 
Moreover, it is known that its worst case error, 
and therefore also the information complexity,
can be expressed in terms of $\lambda$, 
\ie in terms of the squared singular values 
of the univariate problem operator $S_1$, see, e.g., 
Sections 4.2.3 and 5.2 in Novak and Wo{\'z}niakowski~\cite{NW08}.
It turns out that this algorithm, 
applied to the (anti-)~symmetric problem, 
calculates redundant pieces of information. 
Hence, it can not be optimal in this setting.

In preparation for our algorithms, \autoref{sect_antisym} is devoted 
to (anti-) symmetric subspaces in a more 
general fashion than in this introduction.
Moreover, there we study some basic properties.
In \autoref{sect_OptAlgos} we conclude formulae of algorithms 
for linear tensor product problems defined on these subspaces.
We show their optimality in a wide class of algorithms 
and deduce an exact expression for the $n$-th minimal error 
in terms of the squared singular values of $S_1$.
\autoref{theo_opt_algo} summarizes the main results.
Finally, we use this error formula to obtain 
tractability results in \autoref{sect_complexity}
and apply them to wave functions in \autoref{sect_wave}.
\vspace{5pt}

Our results yield that in any case (if we deal with the absolute error criterion)
\begin{gather*}
			n^{\rm asy}(\epsilon,d) \leq n^{\rm sym}(\epsilon,d) \leq n^{\rm ent}(\epsilon,d)
			\quad \text{for every } \epsilon>0 \text{ and all } d\in\N,
\end{gather*}
where for $d=1$ the terms coincide,
since then we do not claim any (anti-) symmetry.
To see that additional (anti-) symmetry conditions may reduce the 
information complexity dramatically consider the simple case of a linear 
operator $S_1$ with singular values $\sigma$ such that $\lambda_1=\lambda_2=1$ 
and $\lambda_j=0$ for $j\geq 3$.
Then the information complexity of the entire tensor product problem
can be shown to be
\begin{gather*}
			n^{\rm ent}(\epsilon,d)=2^d \quad \text{for all } d\in\N \text{ and } \epsilon < 1.
\end{gather*}
Hence, the problem suffers from the curse of dimensionality and is therefore 
\textit{intractable}.
On the other hand, our results show that in the fully symmetric setting we
have \textit{polynomial tractability}, because
\begin{gather*}
			n^{\rm sym}(\epsilon,d)=d+1 \quad \text{for all } d\in\N \text{ and } \epsilon < 1.
\end{gather*}
It can be proved that in this case the complexity of the fully antisymmetric
problem decreases with increasing dimension $d$ and, finally, 
the problem even gets trivial.
In detail, we have
\begin{gather*}
			n^{\rm asy}(\epsilon,d)=\max{3-d,0} \quad \text{for all } d\in\N \text{ and } \epsilon < 1,
\end{gather*}
which yields \textit{strong polynomial tractability}.

Next, let us consider a more challenging problem where 
$\lambda_1=\lambda_2=\ldots=\lambda_m=1$ and 
$\lambda_j=0$ for every $j>m\geq2$. 
For $m=2$ this obviously coincides with the example studied above, 
but letting $m$ increase may tell us more about the structure
of (anti-) symmetric tensor product problems.
In this situation it is easy to check that
\begin{gather*}
			n^{\rm ent}(\epsilon,d)	= m^d 
			\quad \text{and} \quad 
			n^{\rm asy}(\epsilon,d) = \begin{cases}
																		\binom{m}{d}, & d \leq m\\
																		0, 						& d > m,
																\end{cases}
			\quad \text{for every $d\in\N$ and all $\epsilon < 1$.}
\end{gather*}
Since $\binom{m}{d}\geq 2^{d-1}$ for $d\leq\floor{m/2}$,
this means that for large $m$ the complexity in the antisymmetric case 
increases exponentially fast with $d$ up to a certain maximum.
Beyond this point it falls back to zero.
The information complexity in the symmetric setting is much harder to
calculate for this case. 
However, it can be seen that we have polynomial tractability, but
$n^{\rm sym}(\epsilon,d)$ needs to grow at least linearly with $d$
such that the symmetric problem can not be strongly polynomially tractable,
whereas this holds in the antisymmetric setting. 
The entire problem again suffers from the curse of dimensionality.

The reason why antisymmetric problems are that much easier than 
their symmetric counterparts is that from the antisymmetry condition
it follows that $f(x)=0$ if there exist coordinates $j$ and $l$ such that
$x_j=x_l$. 
Another explanation for the good tractability behavior
of antisymmetric tensor product problems
might be the \textit{initial error} $\epsilon_d^{\rm init}$.
For every choice of $\lambda$ it tends to zero as $d$ grows, 
what is not necessarily the case for the corresponding entire 
and the symmetric problem, respectively.
In fact, we have
\begin{gather*}
		\epsilon^{\rm init}_{d,\rm ent} = \epsilon^{\rm init}_{d, \rm sym} = \lambda_1^{d/2}, \quad \text{ whereas} \quad \epsilon^{\rm init}_{d, \rm asy} = \prod_{j=1}^d \lambda_j^{1/2}.
\end{gather*}

For a last illustrative example consider the case $\lambda_1=1$
and $\lambda_{j+1}=j^{-\beta}$ for some $\beta\geq0$ and all $j\in\N$.
That means that we have the two largest singular values $\sigma_1=\sigma_2$ of $S_1$ equal to one.
The remaining series decays like the inverse of some polynomial.
If $\beta=0$ the operator $S_1$ is not compact, 
since $(\lambda_m)_{m\in\N}$ does not tend to zero. 
Hence, all the information complexities are infinite in this case.
For $\beta>0$, any $\delta>0$ and some $C>0$ it is
\begin{gather*}
			n^{\rm ent}(\epsilon,d)\geq 2^{d}, \quad n^{\rm sym}(\epsilon,d) \geq d+1
			\quad \text{and} \quad n^{\rm asy}(\epsilon,d)\leq C \epsilon^{-(2/\beta+\delta)},
			\quad \text{for all } \epsilon<1, d\in\N.
\end{gather*}
Thus, again for the entire problem we have the curse, whereas the
antisymmetric problem is strongly polynomially tractable.
Once more, the symmetric problem can shown to be polynomially tractable.
Note that in this example the antisymmetric case is not trivial,
because all $\lambda_j$ are strictly positive.
If we replace $j^{-\beta}$ by $\log^{-1}(j+1)$ in this example we obtain
(polynomial) intractability even in the antisymmetric setting.

Altogether these examples show that exploiting an a priori knowledge
about (anti-) symmetries of the given tensor product problem can
help to obtain tractability, but it does not make the problem trivial
in general. 
We conclude the introduction with a partial summary 
of our main complexity results.

\begin{thm}
			Let $\lambda=(\lambda_m)_{m\in\N}$ denote the non-increasing sequence 
			of the squared 		singular values of $S_1\colon H_1 \nach G_1$ 
			and assume $\lambda_2>0$.
			Then for the information complexity of (anti-) symmetric 
			linear tensor product problems $S_d$ we obtain the following characterizations:
			\begin{itemize}
					\item The \textit{fully symmetric} problem is strongly 
								polynomially tractable \wrt 
								the normalized error criterion iff $\lambda\in\l_\tau$ 
								for some $\tau>0$ and $\lambda_1>\lambda_2$.
								Furthermore, in the case $\lambda_1\leq 1$ the problem is strongly polynomially tractable 
								\wrt the absolute error criterion iff 
								$\lambda\in\l_\tau$ and $\lambda_2<1$.
					\item The \textit{fully antisymmetric} problem is strongly 
								polynomially tractable \wrt 
								the absolute error criterion iff $\lambda\in\l_\tau$ 
								for some $\tau>0$.
			\end{itemize}
\end{thm}
\noindent In contrast, it is known, see Novak and Wo{\'z}niakowski \cite{NW08}, that
\begin{itemize}
					\item	the \textit{entire} tensor product problem is 
								never (strongly) polynomially tractable \wrt to normalized error criterion.
								Moreover, the problem is strongly 
								polynomially tractable \wrt 
								the absolute error criterion iff $\lambda\in\l_\tau$ 
								for some $\tau>0$ and $\lambda_1<1$.
\end{itemize}


\section{Spaces with (anti-) symmetry conditions}\label{sect_antisym}


Motivated by the example of wave functions in \autoref{sect_Intro},
we mainly deal with \textit{function spaces} in this section.
To this end, we start by defining (anti-) symmetry properties 
for functions which will lead us to orthogonal projections, 
mapping the function space onto its
subspace of (anti-) symmetric functions.
It will turn out that these projections applied to a given basis 
in the tensor product Hilbert function space lead us 
to handsome formulae for orthonormal bases of the subspaces.
In a final remark we generalize our approach 
and define (anti-) symmetry conditions 
for \textit{arbitrary} tensor product Hilbert spaces 
based on the deduced results for function spaces.
\vspace{5pt}

We use a general approach to (anti-) symmetric functions, 
as it can be found in Section~2.5 of Hamaekers~\cite{H09}.
Therefore, for a moment, consider an abstract separable Hilbert space $F$ 
of real-valued functions defined on a domain $\Omega \subset \R^d$. 
In this part of the paper let $d \geq 2$ be fixed. 
The inner product on $F$ is denoted by $\distr{\cdot}{\cdot}_F$.
Moreover, let $I=I(d) \subset \{1,\ldots,d\}$  
be an arbitrary given non-empty subset of coordinates.
Then we define the set
\begin{gather*}
				\S_I = \{ \pi \colon \{1,\ldots,d\} \nach \{1,\ldots,d\} \sep \pi \text{ bijective and } \pi \big|_{\{1,\ldots,d \} \setminus I} = \id\}
\end{gather*}
of all permutations on $\{1,\ldots,d\}$ that leave the complement of $I$ fixed. 
Obviously, the cardinality of this set is given by $\# \S_I = (\# I)!$, 
where $\#$ denotes the number of elements of a set.
For a given $\pi \in \S_I$ we define the mapping
\begin{gather*}
				\pi' \colon \Omega \nach \R^d, \quad x=(x_1,\ldots,x_d) \mapsto \pi'(x)=(x_{\pi(1)}, \ldots, x_{\pi(d)}).
\end{gather*}
To abbreviate the notation we identify $\pi$ and $\pi'$ with each other.

For an appropriate definition of partial (anti-) symmetry of functions $f\in F$ 
we need the following simple assumptions. For every $\pi \in \S_I$ we assume
\begin{enumerate}[label=(A\arabic*)]
				\item \label{A1} $x\in \Omega$ implies $\pi(x)\in\Omega$,
				\item \label{A2} $f\in F$ implies $f(\pi(\cdot)) \in F$ and
				\item \label{A3} there exists $c_{\pi} \geq 0$ (independent of $f$) such that $\norm{f(\pi(\cdot)) \sep F} \leq c_{\pi} \norm{f \sep F}$.
\end{enumerate}
Note that these assumptions always hold if $F$ is a $d$-fold 
tensor product Hilbert space $H_d=H_1 \otimes\ldots\otimes H_1$ equipped with a cross norm,
as described in the examples of the previous section.

Now we call a function $f \in F$ \textit{partially symmetric with respect to $I$} 
(or \textit{$I$-symmetric} for short) if a permutation $\pi\in\S_I$ 
applied to the argument $x$ does not affect the value of $f$. 
Hence,
\begin{gather}\label{sym}
				f(x)=f(\pi(x)) \quad \text{for all} \quad x\in \Omega \quad \text{and every} \quad \pi \in \S_I.
\end{gather}

Moreover, we call a function $f \in F$ 
\textit{partially antisymmetric with respect to $I$}
(or \textit{$I$-antisymmetric}, respectively) 
if $f$ changes its sign by exchanging the variables 
$x_i$ and $x_j$ with each other, where $i,j\in I$.
That is, we have 
\begin{gather}\label{antisym}
				f(x)=(-1)^{\abs \pi} f(\pi(x)) \quad \text{for all} \quad x\in \Omega \quad \text{and every} \quad \pi \in \S_I,
\end{gather}
where $\abs \pi$ denotes the \textit{inversion number} of the permutation $\pi$. 
The term $(-1)^{\abs{\pi}}$ therefore coincides with the \textit{sign}, 
or \textit{parity of $\pi$} and is equal to the determinant 
of the associated permutation matrix.
In the case $\#I=1$ we do not claim any (anti-) symmetry, 
since the set $\S_I = \{ \id \}$ is trivial.
For $I=\{1,\ldots, d\}$ functions $f$ which satisfy 
\link{sym} or \link{antisym}, respectively, are called 
\textit{fully (anti-) symmetric}.

Note that, in particular, formula \link{antisym} yields that 
the value $f(x)$ of (partially) antisymmetric functions $f$ 
equals zero if $x_i=x_j$ with $i \neq j$ and $i,j\in I$. 
For (partially) symmetric functions such an implication does not hold.
Therefore, the (partial) antisymmetry property is a somewhat more restrictive condition 
than the (partial) symmetry property with respect to the same subset $I$.
As we will see in the next sections this will also affect our complexity estimates.

Next, we define the so-called \textit{symmetrizer $\SI_I^F$} and 
\textit{antisymmetrizer $\AI_I^F$ on $F$ with respect to the subset $I$} by
\begin{gather*}
				\SI_I^F \colon F \nach F, \quad f\mapsto \SI_I^F (f) = \frac{1}{\#\S_I} \sum_{\pi \in \S_I} f(\pi(\cdot))
\end{gather*}
and
\begin{gather*}
				\AI_I^F \colon F \nach F, \quad f\mapsto \AI_I^F (f) = \frac{1}{\#\S_I} \sum_{\pi \in \S_I} (-1)^{\abs{\pi}} f(\pi(\cdot)).
\end{gather*}
If there is no danger of confusion we use the notation 
$\SI_I$ and $\AI_I$ instead of $\SI_I^F$ and $\AI_I^F$, respectively.
The following lemma collects together some basic properties which can be proved easily.
For details see the appendix of this paper.

\begin{lemma}\label{projection}
				Both the mappings $P_I\in\{\SI_I, \AI_I\}$ 
				define bounded linear operators on $F$ with $P_I^2=P_I$. 
				Thus, $\SI_I$ and $\AI_I$ provide orthogonal projections 
				of $F$ onto the closed linear subspaces
				\begin{gather}\label{antisymsubspace}
								\SI_I(F) = \{ f \in F \sep f \text{ satisfies } \link{sym} \}  \quad \text{and} 
								\quad	\AI_I(F) = \{ f \in F \sep f \text{ satisfies } \link{antisym} \} 
				\end{gather}
				of all partially (anti-) symmetric functions 
				(\wrt $I$) in $F$, respectively. 
				Hence, 
				\begin{gather}\label{orth_decomp}
								F = \SI_I(F) \oplus (\SI_I(F))^\bot = \AI_I(F) \oplus (\AI_I(F))^\bot.
				\end{gather}				
\end{lemma}

Note that the notion of partially (anti-) symmetric functions 
can be extended to more than one subset $I$.
Therefore, consider two non-empty subsets of coordinates 
$I,J\subset \{1,\ldots,d\}$ with $I\cap J = \leer$.
Then we call a function $f\in F$ 
\textit{multiple partially (anti-) symmetric with respect to $I$ and $J$} 
if $f$ satisfies \link{sym}, or \link{antisym}, 
respectively, for $I$ and $J$.
Since $I$ and $J$ are disjoint we observe that $\pi \circ \sigma = \sigma \circ \pi$
for all $\pi\in \S_I$ and $\sigma\in \S_J$.
Thus, the linear projections $P_I \in \{\SI_I, \AI_I\}$ and 
$P_J \in \{\SI_J, \AI_J\}$ commute on $F$, \ie
$P_I \circ P_J = P_J \circ P_I$.

Further extensions to more than two disjoint subsets of coordinates are possible.
We will restrict ourselves to the case of at most two (anti-) symmetry conditions, 
because in particular wave functions can be modeled as functions
which are antisymmetric with respect to $I$ and 
and $J=I^C$, where $I^C$ denotes the complement of $I$ in $\{1, \ldots, d\}$; 
see, e.g., \autoref{sect_wave} of this paper.
\vspace{5pt}

Up to this point the function space $F$ was an arbitrary separable 
Hilbert space of \mbox{$d$-variate} real-valued functions.
Indeed, for the definition of (anti-) symmetry 
we did not claim any product structure.
On the other hand, 
it is also motivated by applications to consider tensor product function spaces;
see, e.g., Section 3.6 in Yserentant \cite{Y10}.
In detail, it is well-known that so-called spaces of dominated mixed smoothness, 
e.g. $W_2^{(1,\ldots,1)}(\R^{3d})$, can be represented as certain tensor products; 
see Section 1.4.2 in Hansen \cite{H10}. 

Nevertheless, if we take into account such a structure, 
i.e., assume $F=H_d=H_1\otimes \ldots \otimes H_1$ ($d$ times),
where $H_1$ is a suitable Hilbert space of functions $f\colon D\nach\R$,
it is known that we can construct an orthonormal basis (ONB) of $F$
out of a given ONB of $H_1$.
In fact, if $\{\eta_i \sep i\in\N\}$ is an ONB of the underlying Hilbert
function space $H_1$ then the set of all $d$-fold tensor products
$\{\eta_{d,j} = \bigotimes_{l=1}^d \eta_{j_l} \sep j=(j_1,\ldots,j_d)\in\N^d\}$,
\begin{gather*}
			\eta_{d,j}(x) = \prod_{l=1}^d \eta_{j_l}(x_l), \quad x=(x_1,\ldots,x_d)\in D^d,
\end{gather*}
is mutually orthonormal in $H_d$ and forms
a basis.
To exploit this representation 
we start with a simple observation.
 
Let $j\in \N^d$ and $x\in D^d$, as well as a 
non-empty subset $I$ of $\{1,\ldots,d\}$ be arbitrarily fixed.
If we define $\sigma = \pi^{-1} \in S_I$ then
\begin{eqnarray}
				(\AI_I \eta_{d,j})(x) &=& \frac{1}{\#\S_I} \sum_{\pi \in \S_I} (-1)^{\abs{\pi}} \eta_{d,j}(\pi(x)) 
				= \frac{1}{\#\S_I} \sum_{\pi \in \S_I} (-1)^{\abs{\pi}} \prod_{m=1}^d \eta_{j_m}(x_{\pi(m)}) \nonumber\\
				&=& \frac{1}{\#\S_I} \sum_{\pi \in \S_I} (-1)^{\abs{\pi}} \prod_{m=1}^d \eta_{j_{\sigma(m)}}(x_m)
				= \frac{1}{\#\S_I} \sum_{\pi \in \S_I} (-1)^{\abs{\sigma^{-1}}} \eta_{d,\sigma(j)}(x) \label{pi_inside}\\
				&=& \frac{1}{\#\S_I} \sum_{\sigma \in \S_I} (-1)^{\abs{\sigma}} \eta_{d,\sigma(j)}(x). \nonumber
\end{eqnarray}
For simplicity, once again we identified $\pi(j)=\pi(j_1,\ldots,j_d)$ 
with $(j_{\pi(1)},\ldots,j_{\pi(d)})$ for multi-indices $j\in \N^d$.
Obviously, the same calculation can be made 
for $\SI_I$ without the factor $(-1)$. 
Since $x\in D^d$ was arbitrary we obtain
\begin{gather}\label{antisym_basis}
				\SI_I \eta_{d,j} = \frac{1}{\#\S_I} \sum_{\sigma \in \S_I} \eta_{d,\sigma(j)}	
				\quad \text{and} \quad \AI_I \eta_{d,j} = \frac{1}{\#\S_I} \sum_{\sigma \in \S_I} (-1)^{\abs{\sigma}} \eta_{d,\sigma(j)}
				\quad \text{for all} \quad j\in \N^d.
\end{gather}
Note that in general, \ie for arbitrary $j\in\N^d$ and $\sigma\in\S_I$, 
the tensor products $\eta_{d,\sigma(j)}$ and $\eta_{d,j}$ do not coincide,
because taking the tensor product is not commutative in general.
Therefore, $\SI_I$ is not simply the identity on $\{ \eta_{d,j} \sep j \in \N^d \}$.
On the other hand, 
we see that for different $j\in \N^d$ many of the functions 
$\SI_I \eta_{d,j}$ coincide.
Of course the same holds true for $\AI_I \eta_{d,j}$, 
at least up to a factor of $(-1)$.

We will see in the following that for $P_I\in\{\SI_I, \AI_I\}$ 
a linearly independent subset of all projections 
$\{ P_I\eta_{d,j} \sep j\in\N^d \}$ 
equipped with suitable normalizing constants can be used as
an ONB of the linear subspace $P_I(H_d)$ of $I$-(anti-)symmetric
functions in $H_d$.
To this end, we need a further definition.
For fixed $d\geq 2$ and $I\subset\{1,\ldots,d\}$, let us introduce a function
\begin{gather*}
				M_I =M_{I,d} \colon \N^d \nach \{0,\ldots,\#I\}^{\#I}
\end{gather*} 
which counts how often different integers occur in a 
given multi-index $j\in\N^d$ among the subset $I$ of coordinates,
ordered with respect to their rate.
To give an example let $d=7$ and $I=\{1,\ldots,6\}$. 
Then $M_{I,7}$ applied to $j=(12, 4, 4, 12, 6, 4, 4) \in \N^7$ 
gives the $\#I=6$ dimensional vector $M_{I,7}(j) = (3, 2, 1, 0, 0, 0)$, 
because $j$ contains the number ``$4$'' three times
among the coordinates $j_1,\ldots,j_6$, ``$12$'' two times and so on. 
Since in this example there are only three different numbers involved,
the fourth to sixth coordinates of $M_{I,7}(j)$ equal zero.
Obviously, $M_{I}$ is invariant under 
all permutations $\pi\in\S_I$ of the argument.
Thus, 
\begin{gather*}
				M_{I}(j) = M_{I}(\pi(j)) \quad \text{for all} \quad j\in\N^d \quad \text{and} \quad \pi \in \S_I.
\end{gather*}
In addition, since $M_{I}(j)$ is again a multi-index, we see that
$\abs{M_{I}(j)}=\#I$ and $M_{I}(j)!$ are well-defined for every $j\in\N^d$.
With this tool we are ready to state the following assertion
which can be shown using elementary arguments as well as \autoref{projection}; 
see the appendix.

\begin{lemma}\label{lemma_basis}
				Assume $\{\eta_{d,j} \sep j\in \N^d \}$ to be a given orthonormal 
				tensor product basis of the function space $H_d$ 
				and let $\leer \neq I=\{i_1,\ldots,i_{\# I}\} \subset\{1,\ldots,d\}$.
				Moreover, for $P_I\in\{\SI_I,\AI_I\}$ define functions $\xi_j \colon D^d \nach \R$, 
				\begin{gather*}
								\xi_j = \sqrt{\frac{\#\S_I}{M_{I}(j)!}} \cdot P_I(\eta_{d,j}) 
								\quad \text{ for } \quad j \in \N^d.
				\end{gather*}
				Then the set $\{ \xi_k \sep k \in \nabla_d \}$ builds an 
				orthonormal basis of the partially (anti-) symmetric subspace $P_I(H_d)$, 
				where $\nabla_d$ is given by
				\begin{gather}\label{def_nabla}
								\nabla_d = \begin{cases}
															\{ k \in \N^d \sep k_{i_1} \leq k_{i_2} \leq \ldots \leq k_{i_{\#I}}\},& \text{ if } P_I=\SI_I,\\
															\{ k \in \N^d \sep k_{i_1} < k_{i_2} < \ldots < k_{i_{\#I}}\},& \text{ if } P_I=\AI_I.
													\end{cases}
				\end{gather}
\end{lemma}

Observe that in the antisymmetric case the definition of $\xi_j$
for $j\in\nabla_d$ simplifies, since then $M_I(j)!=1$ for all $j\in\nabla_d$.
Moreover, note that in the special case $I=\{1,\ldots,\#I \}$ 
we have 
\begin{gather*}
				P_I(H_d) = P_I \left( \bigotimes_{j\in I} H_1 \right) \otimes \left( \bigotimes_{j\notin I} H_1 \right).
\end{gather*}
That is, we can consider the subspace of $I$-(anti-)symmetric functions 
$f\in H_d$ as the tensor product
of the set of all fully (anti-) symmetric $\#I$-variate functions 
with the $(d-\#I)$-fold tensor product of $H_1$.
Modifications in connection with multiple partially (anti-) symmetric functions are obvious.

Finally, note that \autoref{lemma_basis} also holds 
if the index set $\N$ of the univariate basis
$\{\eta_i \sep i\in\N\}$ is replaced by a more general 
countable set equipped with a total order.
But let us shortly focus on another generalization 
of the previous results.

\begin{rem}[Arbitrary tensor product Hilbert spaces]
Up to now we exclusively dealt with Hilbert \textit{function} spaces.
However, the proofs of \autoref{projection} and \autoref{lemma_basis} yield 
that there are only a few key arguments in connection with 
(anti-) symmetry such that we do not need this restriction.

Starting from the very beginning we need to adapt the definition 
of $I$-(anti-)symmetry due to \link{sym} and \link{antisym}.
Of course it is sufficient to define this property at first only for basis elements.
Therefore, if $E_d=\{\eta_{d,k} \sep k \in \N^d \}$ 
denotes a tensor product ONB of $H_d$
and $\leer \neq I\subset\{1,\ldots,d\}$ is given then
we call an element $\eta_{d,k} = \bigotimes_{l=1}^d \eta_{k_l}$
\textit{partially symmetric with respect to $I$} 
(or \textit{$I$-symmetric} for short), if
\begin{gather*}
			\eta_{d,k} = \eta_{d,\pi(k)} \quad \text{for all} \quad \pi \in \S_I,
\end{gather*}
where $\S_I$ and $\pi(k)=(k_{\pi(1)},\ldots,k_{\pi(d)})$ are defined as above.
Analogously we define \textit{$I$-antisymmetry} 
with an additional factor $(-1)^{\abs{\pi}}$.
Moreover, an arbitrary element in $H_d$ is called $I$-(anti-)symmetric
if in its basis expansion every element 
with non-vanishing coefficient possesses this property.

Next, the \textit{antisymmetrizer} $\AI_I$ is defined as the uniquely 
defined continuous extension of the linear mapping
\begin{gather*}
			\AI_I \colon E_d \nach H_d, \quad \AI_I(\eta_{d,k}) = \frac{1}{\# \S_I} \sum_{\pi \in \S_I} (-1)^{\abs{\pi}} \eta_{d,\pi(k)}
\end{gather*}
from $E_d$ to $H_d$.
Again the \textit{symmetrizer} $\SI_I$ is given in a similar way.
Hence, in the general setting we define the mappings using 
formula~\link{antisym_basis}, which we derived for the special case.
Note that the triangle inequality yields $\norm{P_I} \leq 1$, for $P_I\in\{\SI_I,\AI_I\}$.

Once more we denote the sets of all $I$-(anti-)symmetric elements of $H_d$
by $P_I(H_d)$, where $P_I\in \{\SI_I,\AI_I\}$.
Observe, that this can be justified
since the operators $P_I$ again provide orthogonal projections onto closed linear subspaces
as described in \autoref{projection}.

Finally, also the proof of \autoref{lemma_basis} can be adapted to the
general Hilbert space case. 
\end{rem}


\section{Optimal algorithms}\label{sect_OptAlgos}


In the present section we conclude optimal algorithms
for linear problems defined on (anti-) symmetric
subsets of tensor product Hilbert spaces 
as described in the previous paragraph.
Moreover, we deduce formulae for the $n$-th minimal errors of
these (anti-) symmetric problems and recover the
known assertions for the entire tensor product problem.

\subsection{Basic definitions and the main result}
Throughout the whole section we use the following notation. 
Let $H_1$ be a (infinite dimensional) separable Hilbert space
with inner product $\distr{\cdot}{\cdot}_{H_1}$ and let $G_1$ be
some arbitrary Hilbert space.
Furthermore, assume $S_1 \colon H_1 \nach G_1$ to be a compact linear operator 
between these spaces and consider its singular value decomposition.
That is, define the compact self-adjoint operator
$W_1=S_1^\dagger S_1 \colon H_1 \nach H_1$
and denote its eigenpairs with respect to a non-increasing 
ordering of the eigenvalues by $\{(e_i,\lambda_i) \sep  i\in\N\}$, i.e.
\begin{gather}\label{univariateeigenpairs}
				W_1(e_i) = \lambda_i e_i, \quad 
				\text{and} \quad \distr{e_i}{e_j}_{H_1} = \delta_{i,j} \quad 
				\text{with} \quad \lambda_1 \geq \lambda_2 \geq \ldots \geq 0.
\end{gather}
Then $\lambda=(\lambda_i)_{i\in\N}$ coincides with the sequence of the squared
singular values $\sigma^2=(\sigma_i^2)_{i\in\N}$ of $S_1$ 
and the set $\{ e_i \sep i \in \N\}$ forms an ONB of $H_1$;
see, e.g., Section 4.2.3 in Novak and Wo{\'z}niakowski~\cite{NW08}.
In the following we will refer to $S_1$ as the \textit{univariate problem} 
or \textit{univariate case}.


For $d\geq 2$, let $H_d = H_1 \otimes \ldots \otimes H_1$
be the $d$-fold tensor product space of $H_1$.
This means that $H_d$ is the closure of the set of all 
linear combinations of formal objects
$f=\bigotimes_{l=1}^d f_l$ with $f_l \in H_1$, 
called \textit{simple tensors} or \textit{pure tensors}.
Here the closure is taken with respect to the
inner product in $H_d$ which is defined such that
\begin{gather*}
				\distr{\bigotimes_{l=1}^d f_l}{\bigotimes_{l=1}^d g_l}_{H_d} = \prod_{l=1}^d \distr{f_l}{g_l}_{H_1} \quad \text{for} \quad f_l, g_l \in H_1.
\end{gather*}
With these definitions $H_d$ is also an infinite dimensional 
Hilbert space and it is easy to check that
\begin{gather}\label{basis_eta}
				E_d=\left\{\eta_{d,j} = \bigotimes_{l=1}^d \eta_{j_l} \in H_d \sep j = (j_1,\ldots,d) \in \N^d\right\}
\end{gather}
forms an orthonormal basis in $H_d$ if $\{\eta_i \in H_1 \sep i\in\N\}$ 
is an arbitrary ONB in the underlying space $H_1$.
Similarly, let $G_d = G_1 \otimes \ldots \otimes G_1$, $d$ times, and
define $S_d$ as the tensor product operator
\begin{gather*}
				S_d = S_1 \otimes \ldots \otimes S_1 \colon H_d \nach G_d.
\end{gather*}
In detail, we define the bounded linear operator $\tilde{S}_d \colon E_d \nach G_d$ 
such that for all 
$j\in\N^d$ we have $\tilde{S}_d(\eta_{d,j}) = \tilde{S}_d (\bigotimes_{l=1}^d \eta_{j_l} ) = \bigotimes_{l=1}^d S_1(\eta_{j_l}) \in G_d$. 
Then $S_d$ is assumed to be the uniquely defined linear, continuous extension of $\tilde{S}_d$ from $E_d$ to $H_d$.

We refer to the problem of approximating $S_d \colon H_d \nach G_d$ 
as the \textit{entire $d$-variate problem}.
In contrast, we are interested in the restriction 
$S_d\big|_{P_I(H_d)} \colon P_I(H_d) \nach G_d$ 
of $S_d$ to some {(anti-)} symmetric subspace $P_I(H_d)$ with $P_I\in\{\SI_I,\AI_I\}$
as described in the previous section. 
To abbreviate the notation we denote this restriction again by $S_d$
and refer to it as the \textit{$I$-\mbox{(anti-)}symmetric problem}.

For the singular value decomposition of the entire problem 
operator $S_d$ we consider
the self-adjoint, compact operator
\begin{gather*}
				W_d = {S_d}^\dagger S_d \colon H_d \nach H_d.
\end{gather*}
Its eigenpairs $\{(e_{d,j},\lambda_{d,j}) \sep j=(j_1,\ldots,j_d) \in \N^d\}$
are given by the set of all $d$-fold (tensor) products of the 
univariate eigenpairs \link{univariateeigenpairs} of $W_1$, i.e.,
\begin{gather}\label{Eigenpairs}
				e_{d,j}=\bigotimes_{l=1}^d e_{j_l}
				\quad \text{and} \quad 
				\lambda_{d,j}=\prod_{l=1}^d \lambda_{j_l}
				\quad \text{for} \quad j=(j_1,\ldots,j_d)\in\N^d.
\end{gather}
It is well-known how these eigenpairs can be used to construct
a linear algorithm $A'_{n,d}$ which is optimal for the entire 
$d$-variate tensor product problem.
In detail, $A'_{n,d}$ minimizes the \textit{worst case error}
\begin{gather*}
				e^{\rm wor}(A_{n,d}; H_d) = \sup_{f\in \B(H_d)} \norm{A_{n,d}(f)-S_d(f) \sep G_d}
\end{gather*}
among all adaptive linear algorithms $A_{n,d}$ using $n$ continuous linear functionals. 
Here $\B(H_d)$ denotes the unit ball of $H_d$.
In other words, $A'_{n,d}$ achieves the \textit{$n$-th minimal error}
\begin{gather*}
			e(n,d;H_d) = \inf_{A_{n,d}} e^{\rm wor}(A_{n,d}; H_d).
\end{gather*}

With this notation our main result reads as follows.

\begin{theorem}\label{theo_opt_algo}
				Let $\{(e_m, \lambda_m) \sep m\in \N\}$ denote the eigenpairs 
				of $W_1$ given by \link{univariateeigenpairs}.
				Moreover, for $d>1$ let 
				$\leer \neq I=\{i_1, \ldots, i_{\#I}\} \subset\{1,\ldots,d\}$ and 
				assume $S_d$ to be the linear tensor product problem restricted 
				to the $I$-(anti-)symmetric subspace $P_I(H_d)$, 
				where $P_I\in\{\SI_I,\AI_I\}$, 
				of the $d$-fold tensor product space $H_d$.
				Finally, let $\nabla_d$ be given by \link{def_nabla} and define
				\begin{gather}\label{eigenpairs_sym}
								\{(\xi_{\psi(v)}, \lambda_{d,\psi(v)}) \sep v\in \N\} 
								= \{(\xi_k, \lambda_{d,k}) \sep k \in \nabla_d\}
				\end{gather} 
				by $\xi_k=\sqrt{\#\S_I / M_I(k)!}\cdot P_I(e_{k_1}\otimes \ldots \otimes e_{k_d})$
				and	$\lambda_{d,k}=\prod_{l=1}^d \lambda_{k_l}$, for $k\in\nabla_d$, 
				where $\psi\colon \N \nach \nabla_d$ provides a non-increasing 
				rearrangement of $\{ \lambda_{d,k} \sep k\in \nabla_d\}$.\\				
				Then for every $d>1$ the set \link{eigenpairs_sym} denotes
				the eigenpairs of $W_d\big|_{P_I(H_d)}={S_d}^\dagger S_d$. Thus,
				for every $n \in \N_0$, the linear algorithm 
				$A_{n,d}^*\colon P_I(H_d) \nach P_I(G_d)$,
				\begin{gather}\label{opt_algo}
							A_{n,d}^*f = \sum_{v=1}^n \distr{f}{\xi_{\psi(v)}}_{H_d} \cdot S_d \xi_{\psi(v)},
				\end{gather}
				which uses $n$ linear functionals, is $n$-th optimal 
				for $S_d$ on $P_I(H_d)$ with respect to the worst case setting. 
				Furthermore, it is
				\begin{gather}\label{nth_error}
								e(n,d; P_I(H_d)) 
								= e^{\rm wor}(A_{n,d}^*;P_I(H_d)) 
								= \sqrt{\lambda_{d, \psi(n+1)}}.
				\end{gather}
\end{theorem}

Let us add some remarks on this theorem.
First of all, the sum over an empty index set is 
to be interpreted as zero such that $A_{0,d}^*f \equiv 0$.
Further, note that the worst case error can be attained 
with the function $\xi_{\psi(n+1)}$.
It can be improved neither by non-linear algorithms 
using continuous information,
nor by linear algorithms using adaptive information.
Moreover, observe that the classical entire tensor product
problem is included as the case $\#I=1$, where
we do not claim any (anti-) symmetry.
Then $\nabla_d=\N^d$ and the $\xi_k$'s simply
equal the tensor products $e_{d,k}=\otimes_{l=1}^d e_{k_l}$.
Hence, $A^*_{n,d}=A'_{n,d}$.

The remainder of this section is devoted to the proof of 
the main result \autoref{theo_opt_algo}.

\subsection{Proof of Theorem 1}\label{sect_auxresults}
We start with an auxiliary result which shows that any optimal
algorithm $A^*$ for $S_d$ needs to preserve the 
(anti-) symmetry properties of its domain of definition, 
\ie$A^*f \in P_I(G_d)$ for all $f\in P_I(H_d)$.
The following proposition generalizes Lemma 10.2 
in Zeiser~\cite{Z10} where this assertion was shown 
for the approximation problem, \ie for $S_d=\id$.
A comprehensive proof can be found in the appendix of this paper.

\begin{prop}\label{theo_bestapprox}
				Let $d>1$ and assume $\leer \neq I \subset \{1,\ldots,d\}$. 
				Furthermore, for $X\in\{H,G\}$, let $P_I^X$ denote the (anti-) symmetrizer $P_I\in\{\SI_I,\AI_I\}$ on $X_d$ 
				with respect to $I$ and suppose $A \colon P_I^H(H_d) \nach G_d$ to be 
				an arbitrary algorithm for $S_d$.
				Then, for $g \in H_d$,
				\begin{gather}\label{commute}
								(S_d \circ P_I^H)(g) = (P_I^G \circ S_d)(g),
				\end{gather}
				and for all $f \in P_I^H(H_d)$ it holds
				\begin{gather}\label{bestapprox}
								\norm{S_d f - Af \sep G_d}^2 = \norm{S_d f - P_I^G (A f) \sep G_d}^2 + \norm{Af - P_I^G (A f) \sep G_d}^2.
				\end{gather}
				Hence, an optimal algorithm $A^*$ for $S_d$ preserves (anti-) symmetry, \ie
				\begin{gather*}
								A^*f \in P_I^G(G_d) \quad \text{for all} \quad f \in P_I^H(H_d).
				\end{gather*}
\end{prop}

Beside this qualitative assertion we are interested in explicit error bounds. 
Therefore, the next proposition shows an upper bound on the worst case error
of the algorithm $A_{n,d}^*$ given by \link{opt_algo}.

\begin{prop}[Upper bound]\label{theo_upperbound}
				Under the assumptions of \autoref{theo_opt_algo}
				the worst case error of $A_{n,d}^*$ given by \link{opt_algo} is bounded from above by
				\begin{gather*}
							e^{\rm wor}(A^*_{n,d}; P_I(H_d)) \leq \sqrt{\lambda_{d,\psi(n+1)}}.
				\end{gather*}				
\end{prop}

\begin{proof}
By \autoref{lemma_basis} we have for all $f\in P_I(H_d)$ the unique representation
\begin{gather*}
				f = \sum_{k \in \nabla_d} \distr{f}{\xi_k} \cdot \xi_k.
\end{gather*}
Therefore, the boundedness of $S_d$ together with \link{eigenpairs_sym} implies that
\begin{gather}\label{Sdf}
				S_d f = \sum_{k \in \nabla_d} \distr{f}{\xi_k}_{H_d} \cdot S_d \xi_k = \sum_{v \in \N} \distr{f}{\xi_{\psi(v)}}_{H_d} \cdot S_d \xi_{\psi(v)}
				\quad \text{for every} \quad f \in P_I(H_d).
\end{gather}
Furthermore, in the case $P_I=\AI_I$ it is easy to see that we have
\begin{eqnarray*}
				\distr{S_d \xi_j}{S_d \xi_k}_{G_d} 
				&=& \frac{\#\S_I}{\sqrt{M_I(j)!\cdot M_I(k)!}} \cdot \distr{S_d \AI_I e_{d,j}}{S_d \AI_I e_{d,k}}_{G_d} \\
				&=& \frac{1}{\# \S_I \sqrt{M_I(j)!\cdot M_I(k)!}} \sum_{\pi,\sigma \in \S_I} (-1)^{\abs{\pi}+\abs{\sigma}} \distr{S_d e_{d,\pi(j)}}{S_d e_{d,\sigma(k)}}_{G_d}
\end{eqnarray*}
for $i,j\in \nabla_d$, because of the commutativity of $S_d$ and $\AI_I$ due to \link{commute} in \autoref{theo_bestapprox}.
Obviously, the same calculation can be done in the symmetric case, where $P_I=\SI_I$.
Since $e_{d,\pi(j)}$ and $e_{d,\sigma(k)}$ are 
orthonormal eigenelements of $W_d={S_d}^\dagger S_d\colon H_d\nach H_d$,
see \link{Eigenpairs}, it is 
$\distr{S_d e_{d,\pi(j)}}{S_d e_{d,\sigma(k)}}_{G_d} = \lambda_{d,\pi(j)} \distr{e_{d,\pi(j)}}{e_{d,\sigma(k)}}_{H_d} = \lambda_{d,\pi(j)} \delta_{\pi(j),\sigma(k)}$.
Hence, similar to the proof of the mutual orthonormality of $\{\xi_k \sep k\in\nabla_d\}$ 
for \autoref{lemma_basis} we obtain
\begin{gather}\label{Sdxi_orth}
				\distr{S_d \xi_j}{S_d \xi_k}_{G_d}  = \lambda_{d,j} \delta_{j,k}
			  \quad \text{for all} \quad j,k \in \nabla_d.
\end{gather}
Therefore, we calculate for $n\in \N_0$ and $f\in P_I(H_d)$
\begin{gather*}
				\norm{S_d f - A_{d,n}^* f \sep G_d}^2 = \norm{\sum_{v>n} \distr{f}{\xi_{\psi(v)}}_{H_d} \cdot S_d\xi_{\psi(v)} \sep G_d}^2 = \sum_{v>n} \distr{f}{\xi_{\psi(v)}}_{H_d}^2 \cdot \lambda_{d,\psi(v)}.
\end{gather*}
On the other hand, for $f\in\B(P_I(H_d))$, we have by Parseval's identity 
\begin{gather*}
				1 \geq \norm{f \sep P_I(H_d)}^2 = \norm{f \sep H_d}^2 = \norm{\sum_{v\in\N} \distr{f}{\xi_{\psi(v)}}_{H_d} \cdot \xi_{\psi(v)} \sep H_d}^2 = \sum_{v\in\N} \distr{f}{\xi_{\psi(v)}}_{H_d}^2.
\end{gather*}
Thus, because of the non-increasing ordering of $( \lambda_{d,\psi(v)} )_{v\in\N}$ 
due to the choice of the rearrangement $\psi$, we can estimate the worst case error
\begin{gather*}
				e^{\rm wor}(A_{n,d}^*; P_I(H_d))^2 = \sup_{f\in \B(P_I(H_d))} \norm{S_d f - A_{n,d}^* f \sep G_d}^2 \leq \lambda_{d,\psi(n+1)},
\end{gather*}
as claimed.
\end{proof}

Note that formula \link{Sdxi_orth} in the proof of \autoref{theo_upperbound} 
together with \autoref{lemma_basis} yields that the set \link{eigenpairs_sym}
describes the eigenpairs of the self-adjoint operator
\begin{gather*}
				W_d \big|_{P_I(H_d)} = {S_d}^\dagger S_d \colon P_I(H_d) \nach P_I(H_d)
\end{gather*}
as stated in \autoref{theo_opt_algo}.
Therefore, the upper bound given in \autoref{theo_upperbound} is sharp 
and $A_{n,d}^*$ in \link{opt_algo} is $n$-th optimal, 
due to the general theory; see, e.g., 
Corollary~4.12 in Novak and Wo{\'z}niakowski~\cite{NW08}.
From the general theory it also follows that adaption does not help to
improve this $n$-th minimal error, see \cite[Theorem~4.5]{NW08}, and that
linear algorithms are best possible; see \cite[Theorem~4.8]{NW08}.
Hence, the proof of \autoref{theo_opt_algo} is complete.

Since it seems to be a little bit unsatisfying to refer to these deep results
for the proof of such an easy theorem we refer the reader to the appendix
where a nearly self-contained proof of the remaining facts can be found.
Moreover, there we describe what we mean by adaption in this context.


\section{Complexity}\label{sect_complexity}


In this part of the paper we investigate tractability properties 
of approximating the linear tensor product operator $S_d$
on certain (anti-) symmetric subsets $P_{I}(H_d)=P_{I_d}(H_d)$,
where $P \in \{\SI,\AI\}$ and $\leer \neq I_d \subset\{1,\ldots,d\}$.
Therefore, as usual, we express the $n$-th minimal error 
derived in formula \link{nth_error} in terms
of the \textit{information complexity}, 
\ie the minimal number of information operations 
needed to achieve an error smaller than
a given $\epsilon > 0$,
\begin{gather*}
				n(\epsilon,d;P_I(H_d)) = \min{n\in\N_0 \sep e(n,d; P_I(H_d)) \leq \epsilon}.
\end{gather*}
To abbreviate the notation we write $n^{\rm ent}(\epsilon,d)$ if we deal with the
entire tensor product problem.
Furthermore, as in the introduction, we denote the information complexity 
of the fully (anti-) symmetric problem by $n^{\rm asy}(\epsilon,d)$ and
$n^{\rm sym}(\epsilon,d)$, respectively.

\subsection{Preliminaries}\label{sect_prelim}
From \autoref{theo_opt_algo} we obtain for any $\epsilon>0$ and every $d\in\N$
\begin{gather*}
				n(\epsilon,d;P_I(H_d)) 
				= \min{n\in\N_0 \sep \lambda_{d, \psi(n+1)} \leq \epsilon^2} 
				= \#\left\{k\in\nabla_d\sep \prod_{l=1}^d \lambda_{k_l}>\epsilon^2\right\}
\end{gather*}
by solving \link{nth_error} for $\psi$.
Using this expression we can easily conclude the results for the 
first two problems in the introduction.
There we dealt with the case $\lambda_1=\ldots=\lambda_m=1$ and
$\lambda_j=0$ for $j>m\geq 2$.

Let us recall some common notions of tractability.
If for a given problem the information complexity $n(\epsilon,d)$ increases 
exponentially in the dimension $d$ we say the problem suffers from the 
\textit{curse of dimensionality}. 
That is, there exist constants 
$c>0$ and $C>1$ such that for at least one $\epsilon > 0$ we have
\begin{gather*}
        n(\epsilon,d) \geq c \cdot C^d
\end{gather*}
for infinitely many $d\in\N$. 
More generally, if the information complexity 
depends exponentially on $d$ or $\epsilon^{-1}$ we call the problem 
\textit{intractable}. 
Since there are many ways to measure the lack of 
exponential dependence we distinguish between different 
types of tractability. 
The most important type is \textit{polynomial tractability}.
We say that the problem is polynomially tractable if there exist
constants $C,p>0$, as well as $q\geq0$, such that
\begin{gather*}
				n(\epsilon,d) \leq C \cdot \epsilon^{-p} \cdot d^{q} 
				\quad \text{for all} \quad d\in\N, \epsilon \in(0,1].
\end{gather*}
If this inequality holds with $q=0$, the problem is called 
\textit{strongly polynomially tractable}.
If polynomial tractability does not hold we say the problem
is \textit{polynomially intractable}. 
For more specific definitions and relations 
between these and other classes of tractability 
see, e.g., the monographs of Novak and Wo{\'z}niakowski~\cite{NW08,NW10,NW11}.

In the following we distinguish two cases.
First we consider the \textit{absolute error criterion}, 
where we investigate the dependence of
$n(\epsilon,d; P_I(H_d))$ on $1/\epsilon$ and on the dimension~$d$ 
for every $\epsilon \in (0,1]$
and $d\in\N$. 
Note that without loss of generality we can restrict ourselves to
$\epsilon \leq \min{1,\epsilon_d^{\mathrm{init}}}$ since obviously
$n(\epsilon,d; P_I(H_d))=0$ for all $\epsilon\geq\epsilon_d^{\mathrm{init}}$. 
Here
\begin{gather*}
				\epsilon_d^{\rm init} = e(0,d; P_I(H_d)) = \sqrt{\lambda_{d,\psi(1)}} = \begin{cases}
						\sqrt{\lambda_1^d}, & \text{ if } P=\SI,\\
						\sqrt{\lambda_1^{b_d} \cdot \lambda_1 \cdot \ldots \cdot \lambda_{a_d}}, & \text{ if } P=\AI
				\end{cases}
\end{gather*}
describes the \textit{initial error} of the 
$d$-variate problem on the subspace $P_I(H_d)$ 
where $\psi\colon \N \nach \nabla_d$ again is
a non-increasing rearrangement of the set 
of eigenvalues $\{ \lambda_{d,k} \sep k \in \nabla_d\}$ of
$W_d\big|_{P_I(H_d)}={S_d}^\dagger S_d$
and $b_d = d - a_d$ denotes the number of coordinates without 
(anti-) symmetry conditions in dimension $d$, \ie$a_d = \#I_d$ 
and $b_d = d - \#I_d$, respectively.

Afterwards, we deal with the \textit{normalized error criterion}, 
where we especially investigate the dependence 
of $n(\epsilon' \cdot \epsilon_d^{\rm init},d; P_I(H_d))$
on $1/\epsilon'$ for $\epsilon'\in(0,1)$. 
That is, we search for the minimal number of information operations
needed to improve the initial error 
by a factor $\epsilon'$ less than one.

To avoid triviality we will assume $\epsilon_d^{\rm init}>0$, 
for every $d\in\N$, in both cases, because otherwise we have
strong polynomial tractability by default.
From this assumption it follows
that $\lambda_1>0$, which simply means that
$S_d$ is not the zero operator.
Moreover, note that in the case of antisymmetric problems, 
if the number of antisymmetric coordinates, \ie the set $I=I(d)$, 
grows with the dimension, the condition
$\epsilon_d^{\rm init}>0$ (for every $d\in\N$) even implies that
\begin{gather*}
				\lambda_1 \geq \lambda_2 \geq \ldots > 0.
\end{gather*}
Finally, we always assume $\lambda_2>0$, because otherwise $S_d$
is equivalent to a continuous linear functional 
which can be solved exactly
with one information operation; see Novak and Wo{\'z}niakowski~\cite[p.176]{NW08}.\\
\vspace{5pt}

For the study of tractability for the absolute error criterion 
we use a slightly modified version of Theorem 5.1, \cite{NW08}. 
It deals with the more general situation of arbitrary compact linear operators 
between Hilbert spaces.
In contrast to Novak and Wo\'zniakowski we drop the 
(hidden) condition $\epsilon_d^{\rm init}=1$ 
for the initial error in dimension $d$.
For the sake of completeness a proof can be found in the appendix.
If we denote Riemann's zeta function by $\zeta$ the
assertion reads as follows.

\begin{prop}\label{prop_NW}
			Consider a family of compact linear operators $\{T_d \colon F_d \nach G_d \sep d\in\N\}$
			between Hilbert spaces and the absolute error criterion in the worst case setting. 
			Furthermore, for $d\in\N$ let $(\lambda_{d,i})_{i\in\N}$ denote the non-negative sequence 
			of eigenvalues of ${T_d}^\dagger T_d$ \wrt a non-increasing ordering.
			\begin{itemize}
			\item If $\{T_d\}$ is polynomially tractable with the constants
						$C,p>0$ and $q\geq 0$ then for all $\tau > p/2$ we have
						\begin{gather}\label{sup_condition}
										C_\tau = \sup_{d\in\N} \frac{1}{d^{r}} \left( \sum_{i=f(d)}^\infty \lambda_{d,i}^\tau \right)^{1/\tau} < \infty,
						\end{gather}
						where $r=2q/p$ and $f\colon \N \nach \N$ with 
						$f(d)=\ceil{(1+C) \, d^q}$.\\
						In this case $C_\tau \leq C^{2/p} \, \zeta(2\tau/p)^{1/\tau}$.
			\item If \link{sup_condition} is satisfied for some parameters $r \geq 0$, $\tau >0$ 
						and a	function $f\colon\N \nach \N$ such that 
						$f(d)=\ceil{C\cdot \left(\min{\epsilon_d^{\rm init},1}\right)^{-p}\cdot d^{q}}$, 
						where $C>0$ and $p,q \geq 0$, then the problem is polynomially tractable
						and $n(\epsilon,d)\leq (C+C_\tau^{\tau}) \, \epsilon^{-\max{p,2\tau}} \, d^{\max{q,r\tau}}$
						for every $d\in\N$ and any $\epsilon \in (0,1]$.
%
			\end{itemize}
\end{prop}

Let us add some comments on this result.
Since, clearly,
\begin{gather*} 
				1 \leq \left(\min{\epsilon_d^{\rm init},1}\right)^{-p}
				\quad \text{for all} \quad p\geq0
\end{gather*}
\autoref{prop_NW} provides a characterization
for (strong) polynomial tractability, similar to \cite[Theorem 5.1]{NW08}. 
But, compared with the assertions from the authors of \cite{NW08},
our result yields the essential advantage that the given estimates 
incorporate the initial error $\epsilon_d^{\rm init}$.
Hence, if $\epsilon_d^{\rm init}$ is sufficiently small then we can conclude 
polynomial tractability while ignoring a larger set of eigenvalues in the
summation \link{sup_condition}. 
%

Observe that the first statement does not cover any assertion about the initial error,
since $f(d)\geq2$. 
Hence, it might happen that we have (strong) polynomial tractability 
though the largest eigenvalue $\lambda_{d,1}=(\epsilon_d^{\rm init})^2$ 
tends faster to infinity than any polynomial.
To this end, for $d\in\N$, consider the sequences $(\lambda_{d,m})_{m\in\N}$ given by
\begin{gather*}
				\lambda_{d,1}=e^{2d} \quad \text{and} \quad \lambda_{d,m}=\frac{1}{m}, \quad \text{for} \quad m\geq 2.
\end{gather*}
Here, obviously, the initial error grows 
exponentially fast to infinity, but nevertheless the
second point of \autoref{prop_NW} shows that 
$\{S_d\}$ is strongly polynomially tractable,
since \link{sup_condition} holds with $r=p=q=0$, 
and $C=\tau=2$.
\vspace{5pt}

Let us now return to our $I$-(anti-)symmetric tensor product problems
$S_d$ as defined in \autoref{sect_OptAlgos}.
Therefore, let $\leer \neq I_d=\{i_1, \ldots, i_{\#I_d}\} \subset\{1,\ldots,d\}$ and $P_{I_d}\in\{\SI_{I_d},\AI_{I_d}\}$
for every $d>1$.
We start by using \autoref{prop_NW} to conclude a simple necessary condition for 
(strong) polynomial tractability of $\{S_d\}$ in the worst case setting
\wrt the absolute error criterion.
Recall that $\psi\colon \N \nach \nabla_d$ defines a rearrangement of the parameter set $\nabla_d$ given in \link{def_nabla}.
That is,
\begin{gather}\label{reordered_eigenvalues}
				\{\lambda_{d,\psi(v)} \sep v\in\N\}=\left\{\lambda_{d,k}=\prod_{l=1}^d \lambda_{k_l} \sep k \in \nabla_d \right\}
\end{gather}
denotes the set of eigenvalues of ${S_d}^\dagger S_d$ with respect to a non-increasing ordering, 
see \autoref{theo_opt_algo}.

\begin{lemma}[General necessary conditions]\label{prop_general}
					The fact that $\{S_d\}$ is polynomially tractable
					with the constants $C,p>0$ and $q\geq 0$ implies that 
					$\lambda=(\lambda_m)_{m\in\N} \in \l_\tau$ for all $\tau > p/2$.
					Moreover, for	any such $\tau$ and all $d\in\N$ the following estimate holds:
					\begin{gather*}
								\frac{1}{\lambda_{d,\psi(1)}^\tau} \sum_{k\in\nabla_d} \lambda_{d,k}^\tau 
								\leq (1+C)\, d^q
									+ C^{2\tau/p} \, \zeta\left(\frac{2\tau}{p}\right) \left( \frac{d^{2q/p}}{\lambda_{d,\psi(1)}} \right)^\tau.
					\end{gather*}
\end{lemma}

\begin{proof}
From \autoref{prop_NW} we know that for $\tau > p/2$ and $r=2q/p$ it is
\begin{gather}\label{sup_condition3}
				\sup_{d\in\N} \frac{1}{d^r} \left( \sum_{v=f(d)}^\infty \lambda_{d, \psi(v)}^\tau \right)^{1/\tau} <\infty,
\end{gather}
where the function $f\colon \N \nach \N$ is given by $f(d)=\ceil{(1+C)\,d^q}$.

Note that for the proof of the first assertion 
we only need to consider the case where
all $\lambda_m$ are strictly positive.
Then the condition \link{sup_condition3}, in particular, implies that 
the sum in the brackets converges for every fixed $d\in\N$.
If we denote the subset of indices $j\in\nabla_d$ of the 
$f(d)-1$ largest eigenvalues $\lambda_{d,\psi(v)}$ by $L_d$
then there exists a natural number $s=s(d)\geq d$ such that 
$L_d$ is completely contained in the cube
\begin{gather}\label{cube_Qds}
				Q_{d,s} = \{1,\ldots,s\}^d.
\end{gather}
Hence, we can	crudely estimate the sum from below by
\begin{gather*}
				\sum_{j\in\nabla_d \setminus Q_{d,s}} \lambda_{d,j}^\tau
				\leq \sum_{j\in\nabla_d \setminus L_d} \lambda_{d,j}^\tau
				= \sum_{v=f(d)}^\infty \lambda_{d, \psi(v)}^\tau < \infty.
\end{gather*}
Since $R_{d,s}=\{j=(1,2,\ldots,d-1,m)\in\N^d \sep m > s\}$ is a subset of $\nabla_d \setminus Q_{d,s}$,
independently of the concrete (anti-) symmetrizer $P_{I_d}$, 
where $P\in\{\SI,\AI\}$, we obtain
\begin{gather*}
				(\lambda_{1}\cdot\lambda_{2}\cdot\ldots \cdot\lambda_{d-1})^\tau \sum_{m=s+1}^\infty \lambda_{m}^\tau
				= \sum_{j\in R_{d,s}} \lambda_{d,j}^\tau \leq \sum_{j\in\nabla_d \setminus Q_{d,s}} \lambda_{d,j}^\tau.
\end{gather*}
Thus, for each fixed $d\in\N$ the tail series $\sum_{m=s(d)+1}^\infty \lambda_m^\tau$ is finite, 
which is only possible if $\norm{\lambda \sep \l_\tau}<\infty$.
Hence, $\lambda \in \l_\tau$ is necessary for (strong) polynomial tractability.

Let us turn to the second assertion.
Obviously, \link{sup_condition3} implies the existence of some constant $C_1>0$
such that
\begin{gather*}
				\sum_{v=f(d)}^\infty \lambda_{d,\psi(v)}^\tau \leq C_1 d^{r\tau} \quad \text{for all} \quad d\in\N.
\end{gather*}
Indeed, \autoref{prop_NW} yields that we can take $C_1 = C^{2\tau/p} \zeta(2\tau/p)$.
The rest of the sum can also be bounded easily for any $d\in\N$,
\begin{gather*}
				\sum_{v=1}^{f(d)-1} \lambda_{d,\psi(v)}^\tau \leq \lambda_{d,\psi(1)}^\tau (f(d)-1),
\end{gather*}
due to the ordering provided by $\psi$.
Since $\sum_{k\in\nabla_d} \lambda_{d,k}^\tau = \sum_{v=1}^\infty \lambda_{d,\psi(v)}^\tau$, 
it remains to show that $f(d)-1 \leq (1 + C) d^{q}$
for every $d\in\N$ with $\lambda_{d,\psi(1)}>0$,
which is also obvious due to the definition of $f$.
\end{proof}

Since we know that antisymmetric problems are easier than symmetric problems
we have to distinguish these cases in order to conclude sharp conditions
for tractability.

\subsection{Tractability of symmetric problems (absolute error)}\label{sect_symprob}
Beside the general assertion $\lambda \in \l_\tau$, 
we start with necessary conditions for (strong) polynomial tractability
in the symmetric setting.
By $b_d$ we denote the amount of coordinates without
symmetry conditions in dimension $d$.

\begin{prop}[Necessary conditions, symmetric case]
				Let $\{S_d\}$ be the problem considered in \autoref{prop_general} and assume $P=\SI$.
				\begin{itemize}
								\item If $\{S_d\}$ is polynomially tractable and $\lambda_1 \geq 1$ then $b_d \in \0(\ln d)$.
								\item If $\{S_d\}$ is strongly polynomially tractable and $\lambda_1\geq1$ then $b_d \in \0(1)$ and $\lambda_2<1/\lambda_1$.
				\end{itemize}
\end{prop}
\begin{proof}
Assume $\lambda_1\geq 1$ and let $\tau$ be given by \autoref{prop_general}.
Then, independent of the amount of symmetry conditions,
we have $\lambda_{d,\psi(1)}=\lambda_1^d \geq 1$ and
there exist absolute constants $r\geq 0$ and $C>1$ such that
\begin{gather}\label{stp_estimate}
			\frac{1}{\lambda_1^{\tau d}} \sum_{k\in\nabla_d} \lambda_{d,k}^\tau
			\leq C \, d^r, \quad d\in\N,
\end{gather}
due to \autoref{prop_general}.
In the case of strong polynomial tractability we even have $r=0$.
For $d\geq 2$ we use the product structure of $\lambda_{d,k}$, $k\in\nabla_d$, and 
split the sum on the left 
\wrt the coordinates with and without symmetry conditions. 
Hence, we conclude
\begin{gather}\label{splitting}
			\sum_{k=(h,j)\in\nabla_d} \lambda_{d,k}^\tau 
			= \sum_{j\in \N^{b_d}} \lambda_{b_d,j}^\tau 
					\sum_{\substack{h\in\N^{a_d},\\h_1\leq\ldots\leq h_{a_d}}} \lambda_{a_d,h}^\tau
			= \left( \sum_{m=1}^\infty \lambda_m^\tau \right)^{b_d} 
					\sum_{\substack{h\in\N^{a_d},\\h_1\leq\ldots\leq h_{a_d}}} \lambda_{a_d,h}^\tau, 
			\qquad d=a_d+b_d\geq 2,
\end{gather}
which leads to
\begin{gather*}
			\left( \sum_{m=1}^\infty \left(\frac{\lambda_m}{\lambda_1}\right)^\tau \right)^{b_d} 
					\sum_{\substack{h\in\N^{a_d},\\h_1\leq\ldots\leq h_{a_d}}} \prod_{l=1}^{a_d} 
							\left( \frac{\lambda_{h_l}}{\lambda_1}\right)^\tau
			\leq C \, d^r.
\end{gather*}
In any case the second sum in the above inequality is bounded from below by $1$.
Thus, we conclude that 
$(1+\lambda_2^\tau/\lambda_1^\tau)^{b_d} \leq \left( \sum_{m=1}^\infty \lambda_m^\tau/\lambda_1^\tau \right)^{b_d}$
needs to be polynomially bounded from above.
Since we always assume $\lambda_2>0$ this leads to the claimed bounds on $b_d$ in the
case of (strong) polynomial tractability.

It remains to show that $\lambda_2<1/\lambda_1$ 
is necessary for strong polynomial tractability.
To this end, assume for a moment $\lambda_2\geq 1/\lambda_1$.
Then it is easy to see that 
(independent of the number of symmetry conditions) 
there are at least $1+\floor{d/2}$ different $k\in\nabla_d$ 
such that $\lambda_{d,k}\geq 1$. 
Namely, for $l=0,\ldots,\floor{d/2}$ we can take the first $d-l$ coordinates of $k\in\nabla_d$
equal to one. To the remaining coordinates we assign the value two.

In other words, we have $\lambda_{d,\psi(1+\floor{d/2})} \geq 1$.
On the other hand, strong polynomial tractability implies 
$\sum_{v=\ceil{1+C}}^\infty \lambda_{d,\psi(v)}^\tau \leq C_1$ 
for some absolute constants $\tau,C,C_1>0$ and all $d\in\N$; 
see~\link{sup_condition3}.
Hence, for every $d\geq 2\,\ceil{1+C}$,
\begin{gather*}
			C_1 
			\geq \sum_{v=\ceil{1+C}}^\infty \lambda_{d,\psi(v)}^\tau 
			\geq \sum_{v=\ceil{1+C}}^{1+\floor{d/2}} \lambda_{d,\psi(v)}^\tau 
			\geq \lambda_{d,\psi(1+\floor{d/2})}^\tau (2+\floor{d/2}-\ceil{1+C}) 
			\geq \floor{d/2}+1-\ceil{C},
\end{gather*}
because of the ordering provided by $\psi$.
Obviously, this is a contradiction. 
Thus, we have $\lambda_2<1/\lambda_1$ and the proof is complete.
\end{proof}

Note in passing that the previous argument 
can also be used to show that 
(independent of the number of symmetry conditions) 
the information complexity $n(\epsilon,d)$ 
needs to grow at least linearly in $d$
if we assume $\lambda_2 \geq 1/\lambda_1$.
In particular, we cannot have strong polynomial tractability if $\lambda_1=\lambda_2=1$.

We continue the analysis of $I$-symmetric problems with respect to the absolute error criterion
by proving that the stated necessary conditions are also sufficient for (strong) polynomial tractability.
To this end, we need a rather technical preliminary lemma 
that can be proven by elementary induction arguments. 
For the convenience of the reader we included also this proof in the appendix.
\newpage
\begin{lemma}\label{lemma_symNEW}
				Let $(\mu_m)_{m\in\N}$ be a non-increasing sequence of 
				non-negative real numbers with $\mu_1>0$.
				Then, for all $V\in\N_0$ and every $d\in\N$, it holds
				\begin{gather}\label{estimate_VNEW}
								\sum_{\substack{k\in\N^d,\\1\leq k_1\leq\ldots\leq k_d}} \mu_{d,k}
								\leq \mu_1^d \, d^V \left( 1 + V + \sum_{L=1}^d \mu_1^{-L} 
										 \sum_{ \substack{ j^{(L)}\in\N^L,\\V+2\leq j_1^{(L)}\leq\ldots\leq j_L^{(L)} } } \mu_{L,j^{(L)}} \right).
				\end{gather}
\end{lemma}

Now the sufficient conditions read as follows.
Once again, we denote the number of coordinates 
without symmetry conditions in dimension $d$ by $b_d$.

\begin{prop}[Sufficient conditions, symmetric case]\label{prop_suf_sym}
				Let $\{S_d\}$ be the problem considered in \autoref{prop_general}, assume $P=\SI$ and let $\lambda=(\lambda_m)_{m\in\N}\in \l_{\tau_0}$ for some $\tau_0\in(0,\infty)$.
				\begin{itemize}
								\item If $\lambda_1<1$ then $\{S_d\}$ is strongly polynomially tractable.
								\item If $\lambda_1=1>\lambda_2$ and $b_d \in \0(1)$ then $\{S_d\}$ is strongly polynomially tractable.
								\item If $\lambda_1=1$ and $b_d \in \0(\ln d)$ then $\{S_d\}$ is polynomially tractable.
				\end{itemize}
\end{prop}
\begin{proof}
\textit{Step 1}. We start the proof by exploiting the property $\lambda\in\l_{\tau_0}$.
It is easy to see that the ordering of $(\lambda_m)_{m\in\N}$ implies
\begin{gather*}
				m \lambda_m^{\tau_0} 
				\leq \lambda_1^{\tau_0} + \ldots + \lambda_m^{\tau_0} 
				< \sum_{i=1}^\infty \lambda_i^{\tau_0} = \norm{\lambda \sep \l_{\tau_0}}^{\tau_0} < \infty
\end{gather*}
for any $m\in\N$. 
Hence, there exists some $C_{\tau_0}>0$ such that 
$\lambda_m$ is bounded from above by 
$C_{\tau_0} \cdot m^{-r}$ for every $r\leq 1/\tau_{0}$.
Therefore, there is some index such that for every larger $m\in\N$ we have
$\lambda_m<1$. We denote the smallest of these indices by $m_0$.
Similar to the calculation of Novak and Wo{\'z}niakowski 
in \cite[p.180]{NW08} this leads to
\begin{gather*}
				\sum_{m=m_0}^\infty \lambda_m^\tau 
				\leq (p+1) \lambda_{m_0}^\tau + C_{\tau_0}^\tau \int_{m_0+p}^\infty x^{-\tau r} dx 
				= (p+1) \lambda_{m_0}^\tau + \frac{C_{\tau_0}^\tau}{\tau r - 1} \cdot \frac{1}{(m_0+p)^{\tau r -1}}
\end{gather*}
for every $p\in\N_0$ and all $\tau$ such that $\tau r > 1$.
Thus, in particular, with $r=1/\tau_0$ we conclude
\begin{gather*}
				\sum_{m=m_0}^\infty \lambda_m^\tau 
				\leq (p+1) \lambda_{m_0}^\tau + \frac{1/\tau}{1/\tau_0 - 1/\tau} \left( \frac{C_{\tau_0}^{1/(1/\tau_0 - 1/\tau)}}{m_0+p} \right)^{\tau (1/\tau_0-1/\tau)} \quad \text{for all} \quad \tau>\tau_0, p\in\N_0.
\end{gather*}
Note that for a given $\delta>0$ there exists some constant $\tau_1\geq \tau_0$ 
such that for all $\tau > \tau_1$ it is $1/(1/\tau_0 - 1/\tau) \in (\tau_0,\tau_0+\delta)$.
Hence, if $p\in\N$ is sufficiently large then we obtain for all $\tau > \tau_1$
\begin{align*}
				\sum_{m=m_0}^\infty \lambda_m^\tau 
				&\leq (p+1) \lambda_{m_0}^\tau + \frac{\tau_0+\delta}{\tau} \left( \frac{C_1}{m_0+p} \right)^{\tau (1/\tau_0-1/\tau)} \\
				&\leq (p+1) \lambda_{m_0}^\tau + \frac{\tau_0+\delta}{\tau_1} \left( \frac{C_1}{m_0+p} \right)^{\tau/(\tau_0+\delta)},
\end{align*}
where we set $C_1 = \max{1,C_{\tau_0}^{\tau_0+\delta}}$.
Finally, since $\lambda_{m_0}<1$, both the summands tend to zero as $\tau$ approaches infinity.
In particular, there need to exist some $\tau > \tau_1\geq \tau_0$ such that
\begin{gather*}
				\sum_{m=m_0}^\infty \lambda_m^\tau \leq \frac{1}{2}.
\end{gather*}

\textit{Step 2}. All the stated assertions can be seen using the second point of \autoref{prop_NW}.
Indeed, for polynomial tractability, it is sufficient to show that
\begin{gather}\label{sum_pol}
				\sum_{k\in\nabla_d} \lambda_{d,k}^\tau \leq C d^{r \tau}\quad \text{for all} \quad d\in\N
\end{gather}
and some $C,\tau > 0$ as well as some $r\geq 0$.
If this even holds for $r=0$ we obtain strong polynomial tractability.

In the case $\lambda_1<1$ we can estimate the sum on the left of \link{sum_pol}
from above by $( \sum_{m=1}^\infty \lambda_m^\tau )^d$.
Using Step 1 with $m_0=1$ we conclude $\sum_{k\in\nabla_d} \lambda_{d,k}^\tau \leq 2^{-d}$
for some large $\tau > \tau_0$. 
Hence, the problem is strongly polynomially tractable in this case.

For the proof of the remaining points assume $\lambda_1=1$.
In any case $\sum_{k\in\nabla_1} \lambda_{1,k}^\tau \leq \sum_{m=1}^\infty \lambda_m^{\tau_0} = \norm{\lambda \sep \l_{\tau_0}}^{\tau_0} < \infty$ for all $\tau\geq \tau_0$, because of $\lambda \in \l_{\tau_0}$. 
Therefore, we can assume $d\geq 2$ in the following.
Again we split the sum in \link{sum_pol} with respect to the coordinates with and without symmetry conditions,
i.e., for $d=a_d+b_d\geq 2$ we have
\begin{gather}\label{splitting}
				\sum_{k=(h,j)\in\nabla_d} \lambda_{d,k}^\tau 
				= \sum_{j\in\N^{b_d}} \lambda_{b_d,j}^\tau \sum_{\substack{h\in \N^{a_d},\\h_1\leq\ldots\leq h_{a_d}}} \lambda_{a_d, h}^\tau
				= \left( 1 + \sum_{m=2}^\infty \lambda_{m}^\tau \right)^{b_d} \sum_{\substack{h\in \N^{a_d},\\h_1\leq\ldots\leq h_{a_d}}} \lambda_{a_d, h}^\tau.
\end{gather}

If $\lambda_2<1$ and $b_d$ is universally bounded then
the first factor can be bounded by a constant and the second factor 
can be estimated using \autoref{lemma_symNEW} with $V=0$, $d$ replaced by $a_d$ 
and $\mu$ replaced by $\lambda^\tau$.
It follows that if $\tau$ is large enough we have
\begin{gather*}
				\sum_{\substack{h\in \N^{a_d},\\h_1\leq\ldots\leq h_{a_d}}} \lambda_{a_d, h}^\tau 
				\leq 1 + \sum_{L=1}^{a_d} \sum_{ \substack{ j^{(L)}\in\N^L,\\2\leq j_1^{(L)}\leq\ldots\leq j_L^{(L)} } } \lambda_{L,j^{(L)}}^\tau
				\leq 1 + \sum_{L=1}^{a_d} \left( \sum_{m=2}^\infty \lambda_{m}^\tau \right)^L
				\leq 1 + \sum_{L=1}^\infty 2^{-L} = 2,
\end{gather*}
where we again used Step 1 and the properties of geometric series.
Thus, $\sum_{k\in\nabla_d} \lambda_{d,k}^\tau$ is universally bounded in this case 
and therefore the problem is strongly polynomially tractable.

To prove the last point we argue in the same manner.
Now $b_d \in \0(\ln d)$ yields that the first factor in the splitting \link{splitting}
is polynomially bounded in $d$.
For the second factor we again apply \autoref{lemma_symNEW}, but in this case we set
$V=m_0-2$, where $m_0$ denotes the first index $m\in\N$ such that $\lambda_{m}<1$.
Keep in mind that this index is at least two because of $\lambda_1=1$. 
On the other hand it needs to be finite,
since $\lambda \in \l_{\tau_0}$.
Therefore, the second factor in the splitting \link{splitting} is also polynomially bounded in $d$
due to the same arguments as above.
All in all, this proves \link{sum_pol} and the problem is polynomially tractable in this case.
\end{proof}

We summarize the results obtained for $I$-symmetric tensor product problems
in the following theorem.

\begin{theorem}[Tractability of symmetric problems, absolute error]
			Assume $S_1 \colon H_1 \nach G_1$ to be a compact linear operator 
			between two Hilbert spaces and let $\lambda=(\lambda_m)_{ m\in \N}$ denote the sequence 
			of non-negative eigenvalues of $W_1=S_1^\dagger S_1$ \wrt a non-increasing ordering. 
			Moreover, for $d>1$ let $\leer \neq I_d \subset\{1,\ldots,d\}$.
			Assume $S_d$ to be the linear tensor product problem restricted 
			to the $I_d$-symmetric subspace $\SI_{I_d}(H_d)$ 
			of the $d$-fold tensor product space $H_d$, 
			consider the worst case setting \wrt the absolute error criterion
			and let $\lambda_1\leq 1$.\\
			Then the problem is strongly polynomially tractable 
			if and only if $\lambda \in \l_\tau$ for some $\tau>0$ and
			\begin{itemize}
						\item $\lambda_1<1$, or
						\item $1=\lambda_1>\lambda_2$ and $(d-\#I_d) \in \0(1)$.
			\end{itemize}
			Moreover, the problem is polynomially tractable if and only if $\lambda \in \l_\tau$ for some $\tau>0$ and
			\begin{itemize}
						\item $\lambda_1<1$, or
						\item $\lambda_1=1$ and $(d-\#I_d) \in \0(\ln d)$.
			\end{itemize}
\end{theorem}

\subsection{Tractability of symmetric problems (normalized error)}

Here we briefly focus on the normalized error criterion 
for the $I$-symmetric setting.
Since $(\epsilon_d^{\rm init})^2=\lambda_{d,\psi(1)}=\lambda_1^d$ 
for any kind of symmetric problem, 
this means that we have to investigate the influence of
$d$ and $1/\epsilon'$ on
\begin{align*}
			n(\epsilon' \cdot \epsilon_d^{\rm init}, d; \SI_{I_d}(H_d)) 
			&= \min{n \in \N \sep \lambda_{d,\psi(n+1)} \leq (\epsilon')^2 \lambda_{d,\psi(1)}} \\
			&= \# \left\{ k\in\nabla_d \sep \prod_{l=1}^d \left( \frac{\lambda_{k_l}}{\lambda_1} \right) > (\epsilon')^2 \right\}
			\quad \text{for } \epsilon' \in (0,1), d\in\N.
\end{align*}
Hence, in fact we have to study the information complexity 
of a scaled tensor product problem $S_d'\colon \SI_{I_d}(H_d)\nach G_d$
with respect to the absolute error criterion.
The squared singular values of $S_1'$ equal $\mu = (\mu_m)_{m\in\N}$ with
$\mu_m = \lambda_m / \lambda_1$.
Obviously, we always have $\mu_1=1$. 
Furthermore, $\mu \in \l_\tau$ if and only if $\lambda \in \l_\tau$.
This leads to the following theorem.

\begin{theorem}[Tractability of symmetric problems, normalized error]
			Assume $S_1 \colon H_1 \nach G_1$ to be a compact linear operator 
			between two Hilbert spaces and let $\lambda=(\lambda_m)_{ m\in \N}$ denote the sequence 
			of non-negative eigenvalues of $W_1=S_1^\dagger S_1$ \wrt a non-increasing ordering. 
			Moreover, for $d>1$ let $\leer \neq I_d \subset\{1,\ldots,d\}$.
			Assume $S_d$ to be the linear tensor product problem restricted 
			to the $I_d$-symmetric subspace $\SI_{I_d}(H_d)$ 
			of the $d$-fold tensor product space $H_d$
			and consider the worst case setting \wrt the normalized error criterion.\\
			Then the problem is strongly polynomially tractable if and only if $\lambda \in \l_\tau$ for some $\tau>0$ and $\lambda_1>\lambda_2$ and $(d-\#I_d) \in \0(1)$.\\
			Moreover, $\{S_d\}$ is polynomially tractable if and only if $\lambda \in \l_\tau$ for some $\tau>0$ and $(d-\#I_d) \in \0(\ln d)$.
\end{theorem}

\subsection{Tractability of antisymmetric problems (absolute error)}\label{sect_asy_abs}

We start this subsection with simple sufficient conditions
for strong polynomial tractability.

\begin{prop}[Sufficient conditions, antisymmetric case]\label{prop_suf_asy}
				Let $\{S_d\}$ be the problem considered in \autoref{prop_general}, 
				assume $P=\AI$ and let $\lambda=(\lambda_m)_{m\in\N}\in \l_{\tau_0}$ for some
				$\tau_0\in(0,\infty)$.
				\begin{itemize}
								\item If $\lambda_1<1$ then $\{S_d\}$ is strongly polynomially tractable,
											independent of the number of antisymmetry conditions.
								\item If $\lambda_1 \geq 1$	and if there exist constants 
											$\tau \geq \tau_0$ and $d_0\in\N$ such that 
											for the number of antisymmetric coordinates	$a_d$ in dimension $d$ 
											it holds that
											\begin{gather}\label{suf_condition}
														\frac{\ln{(a_d!)}}{d} \geq \ln(\norm{\lambda \sep \l_\tau}^\tau) \quad \text{for all} \quad d\geq d_0
											\end{gather}
											then the problem $\{S_d\}$ is also strongly polynomially tractable.
				\end{itemize}
\end{prop}

\begin{proof}
Like for the symmetric setting, the proof of these 
sufficient conditions is based on
the second point of \autoref{prop_NW}.
We show that under the given assumptions
\begin{gather*}
				\left( \sum_{v=1}^\infty \lambda_{d,\psi(v)}^\tau \right)^{1/\tau} \leq C < \infty
				\quad \text{for every} \quad d\in\N
\end{gather*}
and some $\tau \geq \tau_0$.
Once again $\psi$ and $\nabla_d$ are given as in \link{reordered_eigenvalues} and \link{def_nabla}, respectively.

Since for $d=1$ there is no antisymmetry condition we have $\psi = \id$ and
\begin{gather*}
				\left( \sum_{v=1}^\infty \lambda_{1,\psi(v)}^\tau \right)^{1/\tau} = \left( \sum_{v=1}^\infty \lambda_v^\tau \right)^{1/\tau} = \norm{ \lambda \sep \l_\tau} \leq \norm{ \lambda \sep \l_{\tau_0}}.
\end{gather*}
Therefore, due to the hypothesis $\lambda \in \l_{\tau_{0}}$ the term for $d=1$ is finite.

Hence, let $d\geq 2$ be arbitrarily fixed.
For $s\in\N$ with $s\geq d$ we define the cubes
$Q_{d,s}$ of multi-indices similar to \link{cube_Qds}.
With this notation we obtain the representation
\begin{gather*}
				\sum_{v=1}^\infty \lambda_{d,\psi(v)}^\tau = \sum_{k \in \nabla_d} \lambda_{d,k}^\tau = \lim_{s \nach \infty} \sum_{k \in \nabla_d \cap Q_{d,s}} \lambda_{d,k}^\tau.
\end{gather*}
Without loss of generality we may reorder the set of coordinates such that 
$I_d = \{i_1,\ldots,i_{a_d}\}=\{1,\ldots,a_d\}$.
That is, we assume partial antisymmetry with respect to the first $a_d$ coordinates. 
Furthermore, we define $U_{a_d,s}= \{ j \in Q_{a_d,s} \sep j_1<j_2<\ldots<j_{a_d} \}$ and set $b_d=d-a_d$.

If $b_d > 0$ then the set of multi-indices under consideration splits into two non-trivial parts:
\begin{gather*}
				\nabla_d \cap Q_{d,s} = U_{a_d,s} \times Q_{b_d,s} \quad \text{for all} \quad s \geq d.
\end{gather*}
Because of the product structure of $\lambda_{d,k}$ ($k\in\nabla_d$) this implies
\begin{gather*}
				\sum_{k=(j,i) \in \nabla_d \cap Q_{d,s}} \lambda_{d,k}^\tau 
				= \left(\sum_{j \in U_{{a_d},s}} \prod_{l=1}^{a_d} \lambda_{j_l}^\tau \right) \left( \sum_{i \in Q_{b_d,s}} \prod_{l=1}^{b_d} \lambda_{i_l}^\tau\right).
\end{gather*}
Since the sequence $\lambda=(\lambda_m)_{m\in\N}$ is an element of $\l_\tau$ 
we can easily estimate the second factor for every $s\geq d$ from above by
\begin{gather}\label{est}
				\sum_{i \in Q_{b_d,s}} \prod_{l=1}^{b_d} \lambda_{i_l}^\tau 
				= \prod_{l=1}^{b_d} \sum_{m=1}^s \lambda_{m}^\tau 
				= \left( \sum_{m=1}^s \lambda_{m}^\tau \right)^{b_d} 
				\leq \left( \sum_{m=1}^\infty \lambda_{m}^\tau \right)^{1/\tau \cdot b_d \cdot \tau } 
				= \norm{\lambda \sep \l_\tau}^{b_d \cdot \tau}.
\end{gather}
To handle the first term we need an additional argument.
Note that the structure of $U_{{a_d},s}$ implies
\begin{gather*}
				\sum_{j \in Q_{{a_d},s}} \prod_{l=1}^{a_d} \lambda_{j_l}^\tau 
				= \sum_{\substack{j \in Q_{{a_d},s}\\\exists k,m: j_k=j_m}} \prod_{l=1}^{a_d} \lambda_{j_l}^\tau + a_d! \sum_{j \in U_{{a_d},s}} \prod_{l=1}^{a_d} \lambda_{j_l}^\tau,
\end{gather*}
which leads to the upper bound
\begin{gather*}
				\sum_{j \in U_{{a_d},s}} \prod_{l=1}^{a_d} \lambda_{j_l}^\tau 
				\leq \frac{1}{a_d!} \sum_{j \in Q_{{a_d},s}} \prod_{l=1}^{a_d} \lambda_{j_l}^\tau
				\leq \frac{1}{a_d!} \norm{\lambda \sep \l_\tau}^{a_d \cdot \tau}, 
\end{gather*}
where we used the same arguments as in \link{est}.
Once again this upper bound does not depend on $s\geq d$.
Hence, due to $d=a_d+b_d$, we conclude
\begin{gather*}
				\sum_{v=1}^\infty \lambda_{d,\psi(v)}^\tau 
				= \lim_{s \nach \infty} \sum_{k \in \nabla_d \cap Q_{d,s}} \lambda_{d,k}^\tau
				\leq \frac{1}{a_d!} \norm{\lambda \sep \l_\tau}^{\tau d}
\end{gather*}
for every choice of $\AI_{I_d}$.
Of course, for every $2\leq d < d_0$ this upper bound is trivially 
less than an absolute constant.
Thus, we can assume $d\geq d_0$. 
Then, due to the hypothesis of the second point we have
$\ln(a_d!) \geq \ln{(\norm{\lambda \sep \l_\tau}^{\tau d})}$,
which implies
\begin{gather*}
				\left( \sum_{v=1}^\infty \lambda_{d,\psi(v)}^\tau \right)^{1/\tau}
				\leq \left( \frac{1}{a_d!} \norm{\lambda \sep \l_\tau}^{\tau d} \right)^{1/\tau}
				\leq 1 \quad \text{for} \quad d\geq d_0.
\end{gather*}
Hence, \link{suf_condition} is sufficient for strong polynomial tractability,
independent of $\lambda_1$.
Therefore it suffices to show that $\lambda_1<1$ implies \link{suf_condition}
in order to complete the proof.
To this end, let $\lambda_1<1$. 
We know from Step 1 in the proof of \autoref{prop_suf_sym} that there 
exists some $\tau \geq \tau_0$ such that
\begin{gather*}
				\norm{\lambda \sep \l_\tau}^\tau = \sum_{m=1}^\infty \lambda_m^\tau \leq \frac{1}{2} < 1.
\end{gather*}
Thus, we see that the right hand side of \link{suf_condition} is negative,
whereas the left hand side is non-negative for every choice of $a_d$.
\end{proof}

We also briefly comment on this result.
First, note that a sequence $\lambda = (\lambda_m)_{m\in\N}$ 
that is not included in any $\l_\tau$-space, $0<\tau<\infty$, 
has to converge to zero more slowly than the inverse of any polynomial, 
i.e., $m^{-\alpha}$ for $\alpha > 0$ arbitrarily fixed. 
Thus, only sequences like $\lambda_m = 1/\ln(m)$ lead to 
polynomial intractability in the fully antisymmetric setting.

Secondly, observe that \link{suf_condition} is quite a weak assumption. 
For example if we have
\begin{gather*}
				a_d \geq \ceil{ \frac{d}{\ln d^\alpha} } \quad \text{with} \quad 0 < \alpha < \frac{1}{\ln (\norm{\lambda \sep \l_\tau}^\tau)}
\end{gather*}
for all sufficiently large $d$ then 
\begin{gather*}
				\frac{\ln(a_d!)}{d}
				\geq \frac{a_d (\ln(a_d)-1)}{d} 
				\geq \frac{1}{\alpha} \cdot \frac{\ln \left( \frac{1}{e\alpha} \cdot \frac{d}{\ln d} \right)}{\ln d} 
				\, \longrightarrow \, \frac{1}{\alpha} > \ln(\norm{\lambda \sep \l_\tau}^\tau), \quad d\nach \infty.
\end{gather*}
If $\alpha$ equals its upper bound, \ie 
$\alpha = 1/\ln(\norm{\lambda \sep \l_\tau}^\tau)$, then
the condition \link{suf_condition} does not hold.
This also shows that assumptions like $a_d = \ceil{ d^\beta }$ with $\beta < 1$
are not sufficient to conclude \link{suf_condition}. 

Note that \autoref{prop_NW} allows us to omit the largest $f(d)-1$ eigenvalues
$\lambda_{d,\psi(v)}$ where $f(d)$ may grow polynomially in $(\epsilon_d^{\rm init})^{-1}$
with $d$. We did not use this fact in the proof of the sufficient conditions.

Let us now turn to the necessary conditions.
We will see that we need a condition similar to~\link{suf_condition}
in order to conclude polynomial tractability if we deal with slowly decreasing 
eigenvalues~$\lambda$.

\begin{prop}[Necessary conditions, antisymmetric case]\label{prop_nec_asy}
			Let $\{S_d\}$ be the problem considered
			in \autoref{prop_general} and assume $P = \AI$.
			Furthermore, let $\{S_d\}$ be polynomially tractable
			with the constants $C,p>0$ and $q\geq 0$.\\
			Then, for $d$ tending to infinity, the initial error 
			$\epsilon_d^{\rm init}$ tends to zero faster than the inverse 
			of any polynomial.
			Furthermore, $\lambda= (\lambda_m)_{m\in\N} \in \l_\tau$ for every $\tau > p/2$ 
			and for all $\delta > 0$ there exists some $d_0 \in \N$ such that
			\begin{gather}\label{bound_a}
						\ln\left(\norm{\lambda \sep \l_\tau}^\tau\right) - \delta 
						\leq \frac{1}{d} \sum_{k=1}^{a_d} \ln \left( \frac{\norm{\lambda \sep \l_\tau}^\tau}{\lambda_k^\tau} \right)
						\quad \text{for all} \quad d\geq d_0.
			\end{gather}
			Thus, we have $\lambda_1<1$ or $\lim_{d\nach \infty} a_d = \infty$.
\end{prop}

\begin{proof}
\textit{Step 1}. 
For the whole proof assume $\tau > p/2$ to be fixed. 
Then \autoref{prop_general} shows that $\lambda \in \l_\tau$.
Moreover, we again use the notation $d=a_d + b_d$, 
where $a_d = \#I_d$ denotes the number of coordinates 
with antisymmetry conditions in dimension $d$.
Similar to the symmetric case we can split the sum of the eigenvalues
such that for all $d\in\N$
\begin{gather*}
			\sum_{k\in\nabla_d} \lambda_{d,k}^\tau = \left( \sum_{m=1}^\infty \lambda_m^\tau \right)^{b_d} \sum_{\substack{j\in\N^{a_d},\\1 \leq j_1<\ldots<j_{a_d}}} \lambda_{a_d,j}^\tau \geq \norm{\lambda \sep \l_\tau}^{\tau b_d} \cdot \lambda_1^\tau \cdot \ldots \cdot \lambda_{a_d}^\tau.
\end{gather*}
Hence, \autoref{prop_general} implies that
\begin{gather}\label{asy_est}
			\left(\frac{\norm{\lambda \sep \l_\tau}^\tau}{\lambda_1^\tau} \right)^{b_d}
			\leq (1 + C)\,d^q
						+ C^{2\tau/p} \zeta\left(\frac{2\tau}{p}\right) \left( \frac{d^{2q/p}}{\lambda_{d,\psi(1)}} \right)^\tau.
\end{gather}
In what follows we will use this inequality to conclude all the stated assertions.

\textit{Step 2}. Here we prove the limit property for 
the initial error $\epsilon_d^{\mathrm{init}}=\sqrt{\lambda_{d,\psi(1)}}$, \ie we need to show that for every fixed polynomial $\P > 0$
\begin{gather}\label{zero_limit}
			\lambda_{d,\psi(1)} \P(d) \longrightarrow 0 \quad \text{if} \quad d \nach \infty.
\end{gather}
Since $\lambda_{d,\psi(1)} = \lambda_1^{b_d} \cdot \lambda_1 \cdot \ldots \lambda_{a_d} \leq \lambda_1^{b_d} \cdot \lambda_1^{a_d}=\lambda_1^d$ 
due to the ordering of $\lambda = ( \lambda_m )_{m\in\N}$ 
we can restrict ourselves to the non-trivial case $\lambda_1 \geq 1$ in the following.
Now assume that there exists a subsequence $(d_k)_{k\in\N}$ of natural numbers as well as some constant $C_0>0$ such that $\lambda_{d_k,\psi(1)} \P(d_k)$ is bounded from below by $C_0$
for every $k\in\N$.
Then for every $d=d_k$ the right hand side of~\link{asy_est} is bounded from above by some polynomial $\P_1(d_k)>0$. 
On the other hand, due to the general condition $\lambda_2 > 0$, 
the term $\norm{\lambda \sep \l_\tau}^\tau / \lambda_1^\tau$ is strictly larger than one.
Thus, it follows that there exists some $C_1>0$ such that
\begin{gather*}
			b_{d_k} \leq C_1 \ln(d_k) \quad \text{for every} \quad k\in\N.
\end{gather*}
Therefore, we have in particular $a_{d_k} = d_k - b_{d_k} \nach \infty$, for $k\nach \infty$.
Moreover, the assumed boundedness of $\lambda_{d_k,\psi(1)} \P(d_k)$ leads to
\begin{gather*}
			C_0 \P(d_k)^{-1} 
			\leq \lambda_{d_k,\psi(1)} 
			\leq \lambda_1^{C_1 \ln(d_k)} \cdot \lambda_1 \cdot \ldots \lambda_{a_{d_k}} 
			= d_k^{C_1 \ln(\lambda_1)} \cdot \lambda_1 \cdot \ldots \lambda_{a_{d_k}}
\end{gather*}
since $\lambda_1 \geq 1$.
As we showed in Step 1 of the proof of \autoref{prop_suf_sym} 
the fact $\lambda \in \l_\tau$ yields the existence 
of some $C_\tau>0$ such that $\lambda_m \leq C_\tau m^{-1/\tau}$
for every $m\in\N$. 
Indeed, this holds for $C_\tau=\norm{\lambda \sep \l_\tau}$, 
which needs to be larger than one because of 
$\lambda_1 \geq 1$.
Hence, 
$\lambda_1^\tau \cdot \ldots \cdot \lambda_{a_{d_k}}^\tau \leq C_\tau^{\tau a_{d_k}} (a_{d_k}!)^{-1}$, 
what implies
\begin{gather*}
			\left( \frac{a_{d_k}}{e}\right)^{a_{d_k}} \leq a_{d_k}! \leq (C_\tau^\tau)^{a_{d_k}} \P_2(d_k), \quad k\in\N
\end{gather*}
for some other polynomial $\P_2 > 0$.
Thus, if $k$ is sufficiently large we conclude
\begin{gather*}
			a_{d_k} \leq a_{d_k} \ln \left( \frac{a_{d_k}}{e\,C_\tau^\tau} \right) \leq \ln(\P_2(d_k)),
\end{gather*}
since $a_{d_k} \nach \infty$ implies $a_{d_k}/(e\, C_\tau^\tau) \geq e$ for $k\geq k_0$.
Therefore, the number of antisymmetric coordinates $a_d$ needs to be
logarithmically bounded from above for every $d$ out of the sequence $(d_k)_{k\geq k_0}$.
Because also $b_{d_k}$ was found to be logarithmically bounded this is a contradiction
to the fact $d_k = a_{d_k} + b_{d_k}$.
Thus, the hypothesis $\lambda_{d_k,\psi(1)} \P(d_k) \geq C_0 > 0$ can not be true
for any subsequence $(d_k)_k$.
In other words it holds \link{zero_limit}.

\textit{Step 3}.
Next we show \link{bound_a}. 
From the former step we know 
that there needs to exist some $d^*\in\N$ such that
$1/\lambda_{d,\psi(1)} \geq 1$ for all $d \geq d^*$.
Hence, \link{asy_est} together with $\tau > p/2$ implies
\begin{gather*}
			\left(\frac{\norm{\lambda \sep \l_\tau}^\tau}{\lambda_1^\tau} \right)^{b_d}
			\leq C_2 \left( \frac{d^{2q/p}}{\lambda_{d,\psi(1)}} \right)^\tau 
			= \frac{C_2 d^{2q\tau/p}}{\lambda_1^{\tau b_d} \cdot \lambda_1^\tau \cdot \ldots \cdot \lambda_{a_d}^\tau}
			\quad \text{for} \quad d\geq d^*,
\end{gather*}
where we set $C_2 = 1 + C + C^{2\tau/p} \zeta(2\tau/p)$.
Therefore, we conclude 
\begin{gather*}
			\frac{1}{C_2 d^{2q\tau/p}} \norm{\lambda\sep \l_\tau}^{\tau d} \leq \prod_{k=1}^{a_d} \frac{\norm{\lambda \sep \l_\tau}^\tau}{\lambda_k^\tau}
\end{gather*}
for all $d\geq d^*$, which is equivalent to
\begin{gather}\label{bound_a2}
			\ln \left( \norm{\lambda \sep \l_\tau}^\tau \right) - \frac{\ln (C_2 \, d^{2q\tau/p})}{d}
			\leq \frac{1}{d} \sum_{k=1}^{a_d} \ln\left( \frac{\norm{\lambda \sep \l_\tau}^\tau}{\lambda_k^\tau}\right).
\end{gather}
Obviously, for given $\delta > 0$, 
there is some $d^{**}$ such that 
$\ln(C_2 \, d^{2q\tau/p})/d<\delta$ for all $d\geq d^{**}$.
Hence, we can choose $d_0=\max{d^*,d^{**}}$ 
in order to obtain \link{bound_a}.

\textit{Step 4}.
It remains to show that $\lambda_1 \geq 1$ implies that
$\lim_{d\nach \infty} a_d$ is infinite.
To this end, note that every summand in \link{bound_a2} is strictly positive.
If we assume for a moment the existence of a subsequence $(d_k)_{k\in\N}$
such that $a_{d_k}$ is bounded for every $k\in\N$ then the right hand side of
\link{bound_a2} is less than some positive constant divided by $d_k$.
Hence, it tends to zero if $k$ approaches infinity.
On the other hand, for large $d$, the left hand side of \link{bound_a2} 
is strictly larger than some positive constant, because of $\lambda_1 \geq 1$ and $\lambda_2>0$. 
This contradiction completes the proof.
\end{proof}

As mentioned before there are examples such that the sufficient condition
\link{suf_condition} from \autoref{prop_suf_asy} is also necessary 
(up to some constant factor) in order to conclude polynomial tractability
in the antisymmetric setting.
Now we are ready to give such an example.

\begin{example}
Consider the situation of \autoref{prop_general} for $P=\AI$
and assume the problem $\{S_d\}$ to be polynomially tractable.
In addition, for a fixed $\tau\in(0,\infty)$, let 
$\lambda=(\lambda_m)_{m\in\N} \in \l_\tau$ be given such that $\lambda_1 \geq 1$
and, moreover, assume that there exist $m_0 \in \N$ such that
\begin{gather}\label{assumption}
				\lambda_m \geq \frac{\norm{\lambda \sep \l_\tau}}{m^{\alpha / \tau}}
				\quad \text{for all} \quad m>m_0 \quad \text{and some} \quad \alpha>1.
\end{gather}
Then we claim that for every $\delta > 0$ there exists $\bar{d}\in\N$ such that
\begin{gather}\label{claimed_bound}
				\left( \frac{1}{\alpha}-\delta \right) \ln \left( \norm{\lambda \sep \l_\tau}^\tau \right) \leq \frac{\ln(a_d!)}{d} \quad \text{for all} \quad d\geq \bar{d}.
\end{gather}
Recall that due to \autoref{prop_suf_asy}, 
for the amount of antisymmetry $a_d$, 
it was sufficient to assume
\begin{gather*}
				\ln \left( \norm{\lambda \sep \l_\tau}^\tau \right) \leq \frac{\ln(a_d)!}{d}
				\quad \text{for every } d \text{ larger than some fixed } d_0\in\N
\end{gather*}
in order to conclude strong polynomial tractability.

Before we prove the claim it might be useful to give a concrete example where
\link{assumption} holds true.
To this end, set $\lambda_m=1/m^2$, $\tau=1$, $\alpha=3$ and $m_0=2$.
Then it is easy to check that $\norm{\lambda \sep \l_\tau}=\zeta(2)=\pi^2/6$ and
obviously we have $\lambda_1=1$.

To see that the claimed inequality \link{claimed_bound} holds true 
we can use \autoref{prop_nec_asy} and, in particular, inequality~\link{bound_a2}.
Since $\lambda_1 \geq 1$ we know that $\lim_{d} a_d = \infty$, \ie$a_d > m_0$
for every $d$ larger than some $d_1\in\N$.
Moreover, note that \link{assumption} is equivalent to
\begin{gather*}
				\ln \left( \frac{\norm{\lambda \sep \l_\tau}^\tau}{\lambda_m^\tau} \right) 
				\leq \alpha \ln(m)
				\quad \text{for all} \quad m>m_0.
\end{gather*}
Hence, if $d\geq d_1$ we can estimate the right hand side of \link{bound_a2} from above by
\begin{gather*}
				\frac{1}{d} \sum_{k=1}^{a_d} \ln \left( \frac{\norm{\lambda \sep \l_\tau}^\tau}{\lambda_k^\tau} \right)
				\leq \frac{m_0}{d}  \cdot \ln \left( \frac{\norm{\lambda \sep \l_\tau}^\tau}{\lambda_{m_0}^\tau} \right) + \frac{\alpha}{d} \sum_{k=m_0+1}^{a_d} \ln(k) 
				\leq \frac{C_\lambda}{d} + \alpha \frac{\ln (a_d!)}{d}.
\end{gather*}
Consequently, this leads to
\begin{gather*}
				\frac{1}{\alpha} \ln \left( \norm{\lambda \sep \l_\tau}^\tau \right) - \frac{C_\lambda + \ln(C_2 d^{2q\tau/p})}{\alpha \cdot d} \leq \frac{\ln (a_d!)}{d}
				\quad \text{for} \quad d\geq \max{d^*,d_1}.
\end{gather*}
Now \link{claimed_bound} follows easily by choosing $\bar{d}\geq \max{d^*,d_1}$ 
large enough such that the negative term 
on the left is smaller than a given $\delta > 0$ 
times $\ln(\norm{\lambda \sep \l_\tau}^\tau)$.
\end{example}

Although there remains a small gap between the necessary 
and the sufficient conditions for the absolute error criterion, 
the most important cases of antisymmetric
tensor product problems are covered by our results.
We summarize the main facts in the next theorem.

\begin{theorem}[Tractability of antisymmetric problems, absolute error]\label{thm_asy_abs}
			Let $S_1 \colon H_1 \nach G_1$ be a compact linear operator 
			between two Hilbert spaces and let $\lambda=(\lambda_m)_{ m\in \N}$ denote the sequence 
			of non-negative eigenvalues of $W_1=S_1^\dagger S_1$ \wrt a non-increasing ordering. 
			Moreover, for $d>1$ let $\leer \neq I_d \subset\{1,\ldots,d\}$.
			Assume $S_d$ to be the linear tensor product problem restricted 
			to the $I_d$-antisymmetric subspace $\AI_{I_d}(H_d)$ 
			of the $d$-fold tensor product space $H_d$
			and consider the worst case setting with respect to the absolute error criterion.\\
			Then for the case $\lambda_1 < 1$ the following statements are equivalent:
			\begin{itemize}
						\item $\{S_d\}$ is strongly polynomially tractable.
						\item $\{S_d\}$ is polynomially tractable.
						\item There exists a universal constant $\tau \in (0,\infty)$ such that $\lambda \in \l_\tau$.
			\end{itemize}
			Moreover, the same equivalences hold true if $\lambda_1\geq 1$ and $\#I_d$ grows linearly with the dimension $d$.
\end{theorem}

Finally, before we continue with the normalized error criterion, 
we want to deduce an exact formula for the complexity in the case of
fully antisymmetric functions.
Hence, we set $I=I_d=\{1,\ldots,d\}$ for every $d \in \N$
and consider
\begin{gather*}
				S_d \colon \AI_I(H_d) \nach G_d
\end{gather*}
in the following.
In this case the set of parameters $\nabla_d$ is given by
\begin{gather*}
				\nabla_d = \{k=(k_1,\ldots,k_d)\in \N^d \sep k_1 < k_2 < \ldots < k_d\}.
\end{gather*}
Thus, to obtain a worst case error less or equal than a given $\epsilon>0$ we need at least
\begin{gather*}
				n^{\rm asy}(\epsilon,d)=n(\epsilon,d; \AI(H_d)) = \# \left\{k\in\nabla_d \sep \lambda_{d,k} = \prod_{l=1}^d \lambda_{k_l} >\epsilon^2\right\}
\end{gather*}
linear functionals for the $d$-variate case.

Due to the ordering of $(\lambda_m)_{m\in\N}$ the largest eigenvalue 
$\lambda_{d,k}$ with $k\in\nabla_d$  is given by the square of 
$\epsilon_d^{\rm init} = \sqrt{\lambda_1 \cdot \lambda_2 \cdot \ldots \cdot \lambda_d}$.
In other words, it is
\begin{gather*}
				n^{\rm asy}(\epsilon,d) = 0 \quad \text{for all} \quad \epsilon \geq \epsilon_d^{\rm init}.
\end{gather*}
To calculate the cardinality also for 
$\epsilon < \epsilon_d^{\rm init}$ let us define
\begin{eqnarray}\label{def_i}
				i_d(\delta^2) = \min{i \in \N \sep \alpha_i=\lambda_i \cdot \lambda_{i+1} \cdot \ldots \cdot \lambda_{i+d-1} \leq \delta^2}, \quad \text{for} \quad \delta>0 \quad \text{and} \quad d\in\N.
\end{eqnarray}
Using this notation we can formulate the following assertion. 
Again the proof can be found in the appendix of this paper.

\begin{prop}[Complexity of fully antisymmetric problems]\label{theo_exact}
				Let $\{S_d\}$ be the problem considered
				in \autoref{prop_general} and assume $P = \AI$ as well as
				$I=I_d=\{1,\ldots,d\}$.\\
				Then for every $\epsilon>0$ the information complexity 
				is given by $n^{\rm asy}(\epsilon,d) = n^{\rm ent}(\epsilon,1)$, 
				if $d=1$, and
				\begin{align}
								&n^{\rm asy}(\epsilon,d) \label{exact}\\
								&\quad = \sum_{l_1=2}^{i_d(\epsilon^2)} \sum_{l_2=l_1+1}^{i_{d-1}(\epsilon^2 / \lambda_{l_1-1})} \ldots \sum_{l_{d-1}=l_{d-2}+1}^{i_{2}(\epsilon^2 / [\lambda_{l_1-1}\cdot\ldots\cdot\lambda_{l_{d-2}-1}])} 
								\left[ n^{\rm ent}\left( \epsilon/\sqrt{\lambda_{l_1-1}\cdot\ldots\cdot\lambda_{l_{d-1}-1}}, 1\right) - l_{d-1} + 1 \right] \nonumber
				\end{align}
				if $d\geq 2$. Here the quantities $i_j$, for $j=2,\ldots,d$, are defined as in \link{def_i}. 
\end{prop}

\begin{rem}
If we define
\begin{gather*}
				\alpha^{(k)}_m=\prod_{l=0}^{k-1} \lambda_{m+l}, \quad m \in \N,
\end{gather*}
for $k\in\N$ and a non-increasing sequence 
$\lambda_1 \geq \lambda_2 \geq \ldots > 0$, 
then we can interpret the quantities 
$i_k(\delta^2)$ as information complexities
of modified univariate problems $S^{(k)}_1$. 
In detail, let $S^{(k)}_1 \colon H_1 \nach G_1$
define a compact linear operator such that
\begin{gather*}
				W^{(k)}_1 = \left( S^{(k)}_1 \right)^\dagger \left(S^{(k)}_1\right) \colon H_1 \nach H_1
\end{gather*}
possesses the eigenvalues $\{\alpha_m^{(k)} \sep m\in\N\}$. 
Then
\begin{gather*}
				n^{\rm ent}(\delta,1)
				= n^{\rm ent}(\delta,1;S^{(k)}_1 \colon H_1 \nach G_1) 
				= i_k(\delta^2) - 1 \quad \text{for all} \quad \delta > 0.
\end{gather*}
Further, note that for $k\geq 2$ the quantities 
$i_k(\epsilon^2/[\lambda_{l_1-1}\cdot\ldots\cdot\lambda_{l_{d-k}-1}])$ are
non-increasing functions in $l_1,\ldots,l_{d-k}$ and $\epsilon$. 
\end{rem}

Out of \autoref{theo_exact} we can conclude bounds on the information complexity. 
If $d\geq 2$ and $\epsilon < \epsilon_d^{\rm init}$ then the sum in \link{exact} contains
at least the term with the index $l_1=2, l_2=3,\ldots,l_{d-1}=d$. 
That is, for any choice of $\lambda$ we get the lower bound
\begin{gather*}
				n^{\rm asy}(\epsilon,d) 
				\geq n^{\rm ent} \left( \epsilon/\sqrt{\lambda_{1}\cdot\ldots\cdot\lambda_{d-1}}, 1 \right) - d + 1,
\end{gather*}
which can be used to show that 
we cannot expect the same nice conditions for (strong) polynomial
tractability as before if we switch from 
the absolute to the normalized error criterion.
We conclude this subsection with a corresponding example.

\begin{example}
Assume $\lambda_m = m^{-2\alpha}$ for all $m \in \N$ and some $\alpha > 0$. Then we need to estimate
\begin{align*}
				n^{\rm ent}\left( \epsilon/\sqrt{\lambda_{1}\cdot\ldots\cdot\lambda_{d-1}}, 1 \right) 
				&= n^{\rm ent}\left( \epsilon \cdot (d-1)!^{\alpha}, 1 \right) \\
				&= \# \left\{ m\in\N \sep m^{-2\alpha} > (\epsilon \cdot (d-1)!^{\alpha})^2 \right\} \\ 
				&= \# \left\{ m\in\N \sep m < \frac{1}{\epsilon^{1/\alpha} \cdot (d-1)!} \right\} \geq \frac{1}{\epsilon^{1/\alpha} \cdot (d-1)!} - 1.
\end{align*}
Therefore,
\begin{gather*}
				n^{\rm asy}(\epsilon,d) 
				\geq \frac{1}{\epsilon^{1/\alpha} \cdot (d-1)!} - d 
				= d \left( \frac{1}{\epsilon^{1/\alpha} \cdot d!} - 1 \right) 
				\quad \text{if} \quad d \geq 2 \quad \text{and} \quad \epsilon < \epsilon_d^{\rm init}.
\end{gather*}
Since in this case the initial error $\epsilon_d^{\rm init}$ 
for the $d$-variate problem equals $1/d!^\alpha$, we need at least
\begin{gather*}
				n^{\rm asy}(\epsilon' \cdot \epsilon_d^{\rm init},d) \geq d \left( \frac{1}{(\epsilon')^{1/\alpha}} - 1 \right)
\end{gather*}
linear functionals to improve the initial error by a factor $\epsilon' < 1$.
Because this bound grows linearly with the dimension the problem is not 
strongly polynomially tractable with respect to the \textit{normalized
error criterion}. 
Nevertheless, the sequence $\lambda$ is an element of $l_{1/\alpha}$, say, 
which implies strong polynomial tractability
for the \textit{absolute error criterion} due to \autoref{thm_asy_abs}.
\end{example}

\subsection{Tractability of antisymmetric problems (normalized error)}\label{sect_normkrit}

Up to now every complexity assertion in this paper 
was mainly based on \autoref{prop_NW} which dealt with the 
general situation of arbitrary compact
linear operators between Hilbert spaces and with the absolute error criterion.
While investigating tractability properties of $I$-symmetric 
problems with respect to the normalized error criterion, 
we were able to use assertions from the absolute error setting.
Since for $I$-antisymmetric problems 
the structure of the initial error is more complicated,
this approach will not work again.
Therefore, we start this subsection with a modified version 
of another known theorem
by Novak and Wo\'zniakowski \cite[Theorem 5.2]{NW08}.

\begin{prop}\label{prop_NW2}
			Consider a family of compact linear operators $\{T_d \colon F_d \nach G_d \sep d\in\N\}$
			between Hilbert spaces and the normalized error criterion in the worst case setting. 
			Furthermore, for $d\in\N$ let $(\lambda_{d,i})_{i\in\N}$ denote the non-negative sequence 
			of eigenvalues of ${T_d}^\dagger T_d$ \wrt a non-increasing ordering. 
			\begin{itemize}
			\item If $\{T_d\}$ is polynomially tractable with the constants
						$C,p>0$ and $q\geq0$ then for all $\tau>p/2$ we have
						\begin{gather}\label{sup_condition2}
										C_\tau = \sup_{d\in\N} \frac{1}{d^r} \left( \sum_{i=f(d)}^\infty \left( \frac{\lambda_{d,i}}{\lambda_{d,1}}\right)^\tau \right)^{1/\tau} < \infty,
						\end{gather}
						where $r=2q/p$ and $f\colon \N \nach \N$ with $f(d) \equiv 1$.\\
						In this case, $C_\tau^\tau \leq 1+C+C^{2\tau/p}\,\zeta(2\tau/p)$.
			\item If \link{sup_condition2} is satisfied for some parameters $r \geq 0$, $\tau >0$ 
						and a	function $f\colon\N \nach \N$ such that $f(d) = \ceil{C\, d^{q}}$, where $C>0$ and $q \geq 0$, 
						then the problem is polynomially tractable
						and $n(\epsilon'\cdot \epsilon_d^{\mathrm{init}},d)\leq (C+C_\tau^\tau)\, (\epsilon')^{-2\tau} d^{\max{q,r\tau}}$.
			\end{itemize}
\end{prop}

Note that this shows that (strong) polynomial tractability 
is characterized by the boundedness of the sum over the normalized eigenvalues, 
were we are allowed to omit the $C d^q$ largest of them.
Of course, our results are equivalent to the assertions given by 
Novak and Wo\'zniakowski~\cite{NW08}, as one can see easily. 
But now the connection between the different error criterions is more obvious. 
From this point of view \autoref{prop_NW2} reads more natural than \cite[Theorem 5.2]{NW08}.
The key is to apply the same proof technique for both the assertions.

Moreover, observe that also the theorem in \cite{NW08} for the normalized error criterion
includes further assertions concerning, e.g., the exponent of strong polynomial tractability. 
Again our proof implies the same results.
\vspace{5pt}

Similar to the former sections we continue 
with an application of \autoref{prop_NW2} 
to our antisymmetric tensor product problems.
To this end, assume $S_1 \colon H_1 \nach G_1$ to be a compact linear operator between two Hilbert spaces and let $\lambda=(\lambda_m)_{m\in\N}$ denote the sequence of non-negative eigenvalues of $W_1=S_1^\dagger S_1$ \wrt a non-increasing ordering. 
Moreover, for $d > 1$ let $\leer \neq I_d = \{1,\ldots,d\}$.
Assume $S_d$ to be the linear tensor product problem restricted to the $I_d$-antisymmetric
subspace $\AI_{I_d}(H_d)$ of the $d$-fold tensor product space $H_d$ and consider the
worst case setting \wrt the normalized error criterion.
Finally, let $b_d$ denote the number of coordinates 
without antisymmetry conditions in dimension $d$, \ie$b_d=d-a_d$, where $a_d=\#I_{d}$
for $d\in\N$.

\begin{prop}[Necessary conditions, antisymmetric case]
			Under these assumptions the fact that $\{S_d\}$ is polynomially tractable
			with the constants $C,p>0$ and $q\geq 0$
			implies that $\lambda \in \l_\tau$ for all $\tau>p/2$.\\ 
			Moreover, for $d$ tending to infinity, 
			$\epsilon_d^{\rm init}$ tends to zero faster 
			than the inverse of any polynomial and $b_d\in\0(\ln d)$.
			Thus, $\lim_{d\nach \infty} a_d/d= 1$.\\
			In addition, if $\{S_d\}$ is strongly polynomially tractable 
			then $b_d \in \0(1)$.
\end{prop}

\begin{proof}
From \autoref{prop_NW2} it follows that there is some $C_1>0$ such that for every $d\in\N$
\begin{gather*}
			\frac{1}{\lambda_{d,\psi(1)}^\tau} \sum_{k\in\nabla_d} \lambda_{d,\psi(v)}^\tau 
			= \sum_{v=1}^\infty \left( \frac{\lambda_{d,\psi(v)}}{\lambda_{d,\psi(1)}} \right)^\tau 
			\leq C_1 d^{2\tau q/p},
\end{gather*}
if $\tau > p/2$. 
Once more the rearrangement function $\psi$ and 
the index set $\nabla_d$ are given as in~\link{reordered_eigenvalues}.
Indeed, the proof of \autoref{prop_NW2} yields that it is sufficient to take
$C_1 = 1+C+C^{2\tau/p} \zeta(2\tau/p)$.
In particular, for $d=1$ it is $\nabla_1=\N$ and $\lambda_{1,k}=\lambda_k$, for $k\in\N$,
such that we have $\psi=\id$ because of the ordering of $\lambda$.
Hence, we conclude
\begin{gather*}
			\norm{\lambda \sep \l_\tau}^\tau = \sum_{k=1}^\infty \lambda_k^\tau \leq C_1 \lambda_1^\tau < \infty.
\end{gather*}
In other words, $\lambda \in \l_\tau$.
Moreover, like with the arguments of Step 1 in the proof of \autoref{prop_nec_asy},
it follows
\begin{gather}\label{est_norm}
			\left( \frac{\norm{\lambda \sep \l_\tau}^\tau}{\lambda_1^\tau} \right)^{b_d}
			\leq C_1 d^{2\tau q/p}, 	\quad d\in\N,
\end{gather}
since $\lambda_{d,\psi(1)}=\lambda_1^{b_d}\cdot \lambda_1 \cdot \ldots \cdot \lambda_{a_d}$
and $\norm{\lambda \sep \l_\tau}^\tau > \lambda_1^\tau$ due to the general assertion $\lambda_2>0$.
Thus, polynomial tractability of $\{S_d\}$ implies $b_d \leq C_2 \ln(d)$ for some $C_2>0$,
\ie$b_d \in \0(\ln d)$.
Therefore, obviously, we have
\begin{gather*}
			1\geq \frac{a_d}{d} 
			= 1- \frac{b_d}{\ln d} \cdot \frac{\ln d}{d} 
			\geq 1-C_2 \cdot \frac{\ln d}{d} \longrightarrow 1, 
			\quad d\nach \infty.
\end{gather*}
The proof that strong polynomial tractability leads to $b_d \in \0(1)$ 
can be obtained using \link{est_norm} with the same arguments as before and $q=0$.
Finally, we need to show the assertion concerning $\epsilon_d^{\rm init}$.
To this end, we refer to Step 2 in the proof of \autoref{prop_nec_asy}.
\end{proof}


\section{Application: wave functions}\label{sect_wave}


During the last decades there has been considerable interest in finding approximations 
of \textit{wave functions}, e.g., solutions of the electronic Schr\"{o}dinger equation. 
Due to the so-called \textit{Pauli principle} of quantum physics only functions with certain
(anti-) symmetry properties are of physical interest. 

In this last section of the present paper we briefly introduce wave functions
and show how our results allow to handle the approximation problem
for such classes of functions.
For a more detailed view, see, 
e.g, Hamaekers~\cite{H09}, Yserentant~\cite{Y10}, or Zeiser~\cite{Z10}.
Furthermore, for a comprehensive introduction to the topic, 
as well as a historical survey, 
we refer the reader to Hunziker and Sigal~\cite{HS00}
and Reed and Simon~\cite{RS78}.

In particular, the notion of multiple partial antisymmetry with respect to 
two sets of coordinates is useful for describing wave functions~$\Psi$.
In computational chemistry such functions 
occur as models which describe quantum states of certain physical $d$-particle systems. 
Formally, these functions depend on $d$ blocks of variables $y_i=(x^{(i)},s^{(i)})$, 
for $i=1,\ldots,d$, which represent the spacial coordinates 
$x^{(i)}=(x_1^{(i)},x_2^{(i)},x_3^{(i)})\in\R^3$ and 
certain additional intrinsic parameters 
$s^{(i)} \in C$ of each particle $y$ within the system.
Hence, rearranging the arguments such that 
$x=(x^{(1)},\ldots,x^{(d)})$ and $s=(s^{(1)},\ldots,s^{(d)})$ yields that
\begin{gather*}
				\Psi \colon (\R^{3})^{d} \times C^d \nach \R, \quad (x,s) \mapsto \Psi(x,s).
\end{gather*}
In the case of systems of electrons one of the most important parameters 
is called \textit{spin} and it can take only two values, i.e., 
$s^{(i)}\in C=\{-\frac{1}{2}, + \frac{1}{2}\}$. 
Due to the Pauli principle 
the only wavefunctions $\Psi$ 
that are physically admissible are those which are antisymmetric 
in the sense that for $I\subset\{1,\ldots,d\}$ and $I^C=\{1,\ldots,d\}\setminus I$
\begin{gather*}
				\Psi(\pi(x),\pi(s))) = (-1)^{\abs{\pi}} \Psi(x,s) \quad \text{for all} \quad \pi \in \S_I \cup \S_{I^C}.
\end{gather*}
Thus, $\Psi$ changes its sign if we replace any particles $y_i$ and $y_j$ 
by each other which posses the same spin, \ie$s^{(i)}=s^{(j)}$.
So, the set of particles, and therefore also the set of spacial coordinates,
naturally split into two groups $I_+$ and $I_-$.
In detail, for wave functions of $d$ particles $y_i$
we can (without loss of generality) assume
that the first $\#I_+$ indices $i$ belong to the group of positive spin, 
whereas the rest of them possess negative spin, 
\ie$I_+=\{1,\ldots,\#I_+\}$ and $I_-=\{\#I_+ + 1,\ldots, d\}$.

In physics it is well-known that some problems, e.g., the electronic Schr\"{o}dinger equation,
which involve (general) wave functions can be reduced to a bunch of similar problems,
where each of them only acts on functions $\Psi_s$ 
out of a certain Hilbert space $F_d=F_d(s)$.
That is,
\begin{gather*}
				\Psi_s =\Psi(\cdot,s) \in F_d=\{f \colon (\R^{3})^{d} \nach \R\}
\end{gather*}
with a given fixed spin configuration $s\in C^d$. 
Of course, every possible spin configuration~$s$ 
corresponds to exactly one choice 
$I_+\subset\{1,\ldots,d\}$ of indices. 
Moreover, it is known that $F_d$ is a Hilbert space which
possesses a tensor product structure.
Therefore, we can model wave functions 
as elements of certain classes of smoothness, 
e.g., $F_d \subset H_d=W_2^{(1,\ldots,1)}(\R^{3d})$, 
as Yserentant~\cite{Y10} recently did,
and incorporate spin properties 
by using the projections of the type 
$\AI = \AI_{I_+} \circ \AI_{I_-}$, 
as defined in \autoref{sect_antisym}.

In particular, \autoref{lemma_basis} yields 
\begin{gather*}
				F_d = \AI(H_d) = \AI_{I_+}(H_{\#I_+}) \otimes \AI_{I_-}(H_{\#I_-})
\end{gather*}
and the system of all
\begin{gather*}
				\tilde{\xi_k} = \sqrt{\#S_{I_+} \cdot \#S_{I_-}} \cdot \AI(\eta_k), \quad k \in \tilde{\nabla}_d,
\end{gather*}
with
\begin{gather*}
				\tilde{\nabla}_d = \{ k=(i,j) \in \N^{\#I_++\#I_-} \sep i_1 < i_2 < \ldots < i_{\#I_+} \text{ and } j_{1} < \ldots < j_{\#I_-} \}
\end{gather*}
builds an orthonormal basis of $F_d=\AI(H_d)$, where
the set $\{\eta_k \sep k=(k_1,\ldots,k_d)\in \N^d \}$ 
is once again assumed to be an orthonormal 
tensor product basis of $H_d=H_1\otimes \ldots \otimes H_1$ 
constructed with the help of
$\{\eta_i \sep i\in\N\}$, an arbitrary orthonormal basis of $H_1$. 

Note that in the former sections the underlying 
Hilbert space $H_1$ always consists of 
univariate functions. 
In contrast wave functions of one particle depend on three variables, 
but we want to stress the point that this is just a formal issue.
However, this approach radically decreases the degrees of freedom 
and improves the solvability of certain problems $S_d$ 
like the approximation problem, \ie$S_1=\id\colon H_1 \nach G_1$, 
considered in connection with the electronic Schr\"{o}dinger equation.

\autoref{theo_opt_algo} then provides an algorithm which is
optimal for the $G_d$-approximation of 
$d$-particle wave functions in $F_d$
with respect to all linear algorithms 
that use at most $n$ continuous linear functionals.
Moreover, the error can be calculated 
exactly in terms of the squared singular values 
$\lambda = (\lambda_m)_{m\in\N}$ of $S_1$.

Furthermore, it is possible to prove 
a modification of \autoref{thm_asy_abs} 
for problems dealing with wave functions.
In fact, for the mentioned approximation problem 
polynomial tractability as well as strong polynomial tractability are
equivalent to the fact that the sequence $\lambda$ 
of the squared singular values
of the univariate problem belong to some $\l_\tau$-space
if we consider the absolute error criterion.
The reason is that all the assertions in \autoref{sect_asy_abs}
can be easily extended to the multiple partially antisymmetric case.
In detail, if we denote the number of 
antisymmetric coordinates $x^{(i)}$ within
each antisymmetry group $I_d^{m}\subset\{1,\ldots,d\}$ 
by $a_{d,m}$, $m=1,\ldots,M$,
then the constraint $a_d + b_d = d$ extends to 
$a_{d,1}+\ldots+a_{d,M}+b_d=d$.
Here $b_d$ again denotes the number of coordinates 
without any antisymmetry condition.
In conclusion, the sufficient condition \link{suf_condition} in
\autoref{prop_suf_asy} transfers to
\begin{gather*}
			\frac{1}{d} \sum_{m=1}^M \ln(a_{d,m}!) \geq \norm{\lambda \sep \l_\tau}^\tau, 
			\quad \text{for all} \quad d\geq d_0,
\end{gather*}
which is always satisfied in the case of wave functions, 
since then $M=2$ and at least the cardinality $a_{d,m}$ of one of the groups of the same spin needs to grow linearly with the dimension~$d$.

\section*{Acknowledgments}
The author thanks E. Novak and H. Wo{\'z}niakowski for 
their valuable comments on this paper.
\addcontentsline{toc}{chapter}{Acknowledgments}

\bibliographystyle{acm}
\bibliography{Bibliography}   
\addcontentsline{toc}{chapter}{References}

\newpage
\section*{Appendix}
\addcontentsline{toc}{chapter}{Appendix}


\subsection*{Proof of \autoref{projection} in \autoref{sect_antisym}}


We show \autoref{projection} in the case of function spaces.
\begin{proof}
Obviously, $P \in \{\SI_I, \AI_I\}$ is well-defined due to the assumptions \ref{A1} and \ref{A2}. 
The linearity directly follows from the definition and, using \ref{A3}, 
the operator norm is bounded by $\max{c_\pi \sep \pi \in \S_I}$.

To show that the operators are idempotent, \ie$P^2=P$, 
we first prove that $\AI_I (f)$ satisfies~\link{antisym} for every $f \in F$. 
Therefore, we use the representation
\begin{gather*}
				(\AI_I(f))(\pi(x))
				= \sum_{\sigma \in \S_I} (-1)^{\abs{\sigma}} f(\sigma(\pi (x))) 
				= \sum_{\lambda \in \S_I} (-1)^{\abs{\lambda} + \abs{\pi}} f(\lambda(x))
				= (-1)^{\abs{\pi}} (\AI_I(f))(x)
\end{gather*}
for every fixed $\pi\in \S_I$. 
Here we imposed $\lambda = \sigma \circ \pi \in \S_I$ and used 
\begin{gather*}
				\abs{\lambda \circ \pi^{-1}} = \abs{\lambda} + \abs {\pi^{-1}} = \abs{\lambda} + \abs {\pi}.
\end{gather*}
Hence, we have $\AI_I(F) \subset \{ f \in F \sep f \text{ satisfies } \link{antisym} \}$.
In a second step, it is easy to check that for every function $g\in F$ which satisfies \link{antisym}
it is $\AI_I(g)=g$. Thus, $\{ f \in F \sep f \text{ satisfies } \link{antisym} \} \subset \AI_I(F)$
and $\AI_I$ is a projector onto $\AI_I(F)$.

Since the same arguments also apply for the symmetrizer $\SI_I$ this shows \link{antisymsubspace}, 
as well as $P^2=P$ for $P\in\{\SI_I, \AI_I\}$.
Because of the boundedness of the operators the subsets $\SI_I(F)$ and $\AI_I(F)$ 
are closed linear subspaces of $F$ and we obtain the orthogonal
decompositions
\begin{gather*}
				F = \SI_I(F) \oplus (\SI_I(F))^\bot = \AI_I(F) \oplus (\AI_I(F))^\bot,
\end{gather*}
where the $\bot$ denotes the orthogonal complement with respect to 
$\distr{\cdot}{\cdot}_F$, \ie the image of the projectors $(\id - \SI_I)$ and $(\id-\AI_I)$, respectively.
\end{proof}

The proof of \autoref{projection} in the case of arbitrary
tensor product Hilbert spaces works exactly in the same way.


\subsection*{Proof of \autoref{lemma_basis} in \autoref{sect_antisym}}


We prove \autoref{lemma_basis} in the case of function spaces.
For the case of arbitrary tensor product Hilbert spaces
only slight modifications are needed.
Indeed, the only difference is the conclusion of 
formula~\link{formula_coeff} in Step 2.
In the general setting this simply follows from our definitions.

\begin{proof}
\textit{Step 1}. 
We start by proving orthonormality.
Therefore, let us recall \link{antisym_basis}. 
To abbreviate the notation further, we suppress
the index $H_d$ at the inner products $\distr{\cdot}{\cdot}_{H_d}$ 
in this proof. For $P_I=\AI_I$ and $j,k \in \nabla_d$ easy calculus yields
\begin{align*}
				\distr{\xi_j}{\xi_k} 
				&= \frac{\# \S_I}{\sqrt{M_{I}(j)! \cdot M_{I}(k)!}} \distr{\AI_I (\eta_{d,j})}{\AI_I (\eta_{d,k})} \\
				&= \frac{1}{\# \S_I \sqrt{M_{I}(j)! \cdot M_{I}(k)!}} \sum_{\pi,\sigma \in \S_I} (-1)^{\abs{\pi}+\abs{\sigma}} \distr{\eta_{d,\pi(j)}}{\eta_{d,\sigma(k)}}.
\end{align*}
Of course, up to the factor controlling the sign, the same is true for the case $P_I=\SI_I$.
Assume now there exists $m\in \{1,\ldots,d\}$ such that $j_m \neq k_m$. 
Then the ordering of $j,k\in\nabla_d$ implies that
$\pi(j)\neq\sigma(k)$ for all $\sigma,\pi \in \S_I$,
since $\pi$ and $\sigma$ leave the coordinates $m\in I^C$ fixed.
Hence, we conclude that we have
$\pi(j) = \sigma(k)$ only if $j=k$.

At this point we have to distinguish the antisymmetric and the symmetric case.
For $P=\AI_I$ the only way to conclude $\pi(j) = \sigma(k)$ is to claim $j=k$ and $\pi=\sigma$.
Furthermore, we see that in the antisymmetric case we have $M_{I}(j)!=1$ for all $j\in\nabla_d$, 
since all coordinates~$j_l$, where $l\in I$, differ.
Therefore, in this case the last inner product coincides with 
$\delta_{j,k} \cdot \delta_{\pi,\sigma}$ because of the mutual orthonormality
of $\{\eta_{d,j} \sep j\in \N^d \}$.
Hence,
\begin{gather*}
				\distr{\xi_j}{\xi_k} = \frac{1}{\# \S_I} \sum_{\pi \in \S_I} (-1)^{2\abs{\pi}} \delta_{j,k} = \delta_{j,k} \quad \text{for all} \quad j,k\in\nabla_d
\end{gather*}
as claimed. 

So, let us consider the case $P_I=\SI_I$ and $j=k\in \nabla_d$, 
because we already saw that otherwise $\distr{\xi_j}{\xi_k}$ equals zero.
Then for fixed $\sigma \in \S_I$ there are $M_{I}(j)!$ different permutations
$\pi \in \S_I$ such that $\pi(j) = \sigma(j)$.
This leads to
\begin{gather*}
				\distr{\xi_j}{\xi_j} = \frac{1}{\# \S_I \cdot M_{I}(j)!} \sum_{\sigma \in \S_I} M_{I}(j)! = 1
\end{gather*}
and completes the proof of orthonormality.

\textit{Step 2}. It remains to show that 
the span of $\{\xi_k \sep k\in \nabla_d\}$ 
is dense in $P_I(H_d)$ for $P\in \{\SI_I,\AI_I\}$.
To this end, note that every multi-index $j\in\N^d$ 
can be represented by a uniquely defined multi-index $k\in \nabla_d$
and exactly $M_I(k)!$ different permutations $\pi\in\S_I$
such that $j=\pi(k)$.

Now assume $f \in \AI_I(H_d)$, \ie$f\in H_d$ satisfies \link{antisym}. 
Then the expansion of $f(\pi(x))$ with respect to the basis functions 
$\{\eta_{d,j} \sep j\in\N^d\}$ in $H_d$ yields
for $\pi \in \S_I$
\begin{gather*}
				(-1)^{\abs{\pi}} \sum_{j\in \N^d} \distr{f}{\eta_{d,j}} \cdot \eta_{d,j}(x) 
				= \sum_{k\in \N^d} \distr{f}{\eta_{d,k}} \cdot \eta_{d,k}(\pi(x)) \quad \text{for every} 
				\quad x\in D^d.
\end{gather*}
Similar to the arguments used in \link{pi_inside} we have 
$\eta_{d,k}(\pi(x)) = \eta_{d,\sigma(k)}(x)$ with $\sigma = \pi^{-1}$. 
Therefore, we conclude for $x\in D^d$
\begin{gather*}
				\sum_{j\in \N^d}  \left( (-1)^{\abs{\pi}} \cdot \distr{f}{\eta_{d,j}} \right) \cdot \eta_{d,j}(x) 
				= \sum_{k\in \N^d} \distr{f}{\eta_{d,\pi(\sigma(k))}} \cdot \eta_{d,\sigma(k)}(x) 
				= \sum_{j\in \N^d} \distr{f}{\eta_{d,\pi(j)}} \cdot \eta_{d,j}(x).
\end{gather*}
Because the expansion is uniquely defined we get
\begin{gather}\label{formula_coeff}
				(-1)^{\abs{\pi}} \cdot \distr{f}{\eta_{d,j}} 
				= \distr{f}{\eta_{d,\pi(j)}} \quad \text{for all} \quad j\in\N^d \quad \text{and} \quad \pi \in \S_I.
\end{gather}
Using the observations from the beginning of this step we can decompose 
the basis expansion of $f\in \AI_I(H_d)\subset H_d$ and use the derived 
formula~\link{formula_coeff} to get
\begin{gather*}
				f = \sum_{j\in\N^d} \distr{f}{\eta_{d,j}} \eta_{d,j} 
				= \sum_{k\in \nabla_d} \sum_{\sigma \in \S_I}\frac{ \distr{f}{\eta_{d,\sigma(k)}} \eta_{d,\sigma(k)} }{M_I(k)!} 
				= \sum_{k\in \nabla_d} \frac{1}{M_I(k)!} \sum_{\sigma \in \S_I} (-1)^{\abs{\sigma}} \distr{f}{\eta_{d,k}} \eta_{d,\sigma(k)}.
\end{gather*}
Now \link{antisym_basis} yields that
\begin{gather*}
				f = \sum_{k\in \nabla_d} \sqrt{\frac{\# \S_I}{M_I(k)!}} \cdot \distr{f}{\eta_{d,k}} \cdot \sqrt{\frac{\# \S_I}{M_I(k)!}} \cdot \AI_I(\eta_{d,k}).
\end{gather*}
Furthermore, summing up \link{formula_coeff} with respect to $\pi$ leads to
\begin{gather*}
				\distr{f}{\eta_{d,k}} 
				= \frac{1}{\# \S_I} \sum_{\pi \in \S_I} (-1)^{\abs{\pi}} \distr{f}{\eta_{d,\pi(k)}} 
				=  \distr{f}{ \frac{1}{\# \S_I} \sum_{\pi \in \S_I} (-1)^{\abs{\pi}} \eta_{d,\pi(k)}} 
				=  \distr{f}{\AI_I (\eta_{d,k})}, 
\end{gather*}
for $k\in \nabla_d$, such that finally $f\in \AI_I(H_d)$ possesses the following representation
\begin{gather*}
				f = \sum_{k \in \nabla_d} \distr{f}{\xi_k} \cdot \xi_k,
\end{gather*}
since $\xi_k=\sqrt{\# \S_I / M_I(k)!} \cdot \AI_I(\eta_{d,k})$ per definition. 
This proves the assertion for the case $P_I=\AI_I$.
The remaining case $P_I=\SI_I$ can be treated in the same way.
\end{proof}


\subsection*{Proof of \autoref{theo_bestapprox} in \autoref{sect_auxresults}}


\begin{proof}
The proof is organized as follows. 
First we show that the problem operator $S_d$ and 
the (anti-) symmetrizer $P_I$ commute on $H_d$, \ie it holds \link{commute}.
In a second step we conclude \link{bestapprox} out of this. 
The (anti-) symmetry of $A^*f$ for an optimal algorithm $A^*$ then follows immediately.

\textit{Step 1}. 
Assume $E_d=\{\eta_{d,j} \sep j\in\N^d\}$ to be an arbitrary tensor product ONB of $H_d$,
as defined in \link{basis_eta}.  
Then, for fixed $j\in\N^d$, formula \link{antisym_basis} 
and the structure of the linear tensor product operator 
$S_d=S_1 \otimes \ldots \otimes S_1$ yields in the case $P_I=\AI_I$
\begin{eqnarray*}
				S_d(\AI_I^H (\eta_{d,j})) 
				&=& S_d\left(\frac{1}{\# \S_I} \sum_{\pi \in \S_I} (-1)^{\abs{\pi}} \bigotimes_{l=1}^d \eta_{j_{\pi(l)}}\right)\\ 
				&=& \frac{1}{\# \S_I} \sum_{\pi \in \S_I} (-1)^{\abs{\pi}} \bigotimes_{l=1}^d S_1(\eta_{j_{\pi(l)}})
				= \AI_I^G (S_d (\eta_{d,j})).
\end{eqnarray*} 
Obviously, the same is true for $P_I=\SI_I$.
Hence, it holds \link{commute} at least on the set of basis elements $E_d$ of $H_d$.
Because of the representation
\begin{gather*}
				g = \sum_{j\in\N^d} \distr{g}{\eta_{d,j}}_{H_d} \cdot \eta_{d,j}, \quad \quad g\in H_d,
\end{gather*}
as well as the linearity and boundedness of the operators 
$P_I^H, P_I^G$ and $S_d$ we can extend the relation \link{commute}
from $E_d$ to the whole space $H_d$.

\textit{Step 2}. Now let $f\in P_I^H(H_d)$ and let $Af$ denote 
an arbitrary approximation of $S_d f$.
Then $S_d f = S_d(P_I^H f) = P_I^G(S_d f)$, due to Step 1.
Using the fact that $P_I^G$ provides an orthogonal projection 
onto $P_I^G(G_d)$, see \link{orth_decomp}, we obtain \link{bestapprox},
\begin{eqnarray*}
				\norm{S_d f - Af \sep G_d}^2 
				&=& \norm{P_I^G (S_d f) - [P_I^G (Af) + (\id^G - P_I^G)(Af)] \sep G_d}^2 \\
				&=& \norm{P_I^G (S_d f - Af) \sep G_d}^2 + \norm{(\id^G - P_I^G)(Af) \sep G_d}^2 \\
				&=& \norm{S_d f - P_I^G(Af) \sep G_d}^2 + \norm{Af - P_I^G(Af) \sep G_d}^2,
\end{eqnarray*}
as claimed.
\end{proof}


\subsection*{A self-contained proof of \autoref{theo_opt_algo} in \autoref{sect_OptAlgos}}


In order to deduce a lower bound on the $n$-th minimal error of
approximating $S_d$ on (anti-) symmetric subspaces 
$P_I(H_d)$, where $P\in\{\AI,\SI\}$, let us define 
the classes of algorithms under consideration.

An algorithm $A_{n,d}$ for $S_d\colon F_d=P_I(H_d)\nach G_d$ 
which uses $n$ pieces of information is 
modeled as a mapping $\phi \colon \R^n \nach G_d$ 
and a function $N \colon F_d \nach \R^n$ such that 
$A_{n,d} = \phi \circ N$. 
In detail, the information map $N$ is given by
\begin{gather}\label{non-adapt}
        N(f)=\left( L_1(f), L_2(f), \ldots, L_n(f) \right), \qquad f\in F_d,
\end{gather}
where $L_j \in \Lambda$. Here we distinguish certain classes of 
information operations $\Lambda$. In one case we assume that we can 
compute continuous linear functionals. 
Then $\Lambda=\Lambda^{\rm all}$ coincides with~$F_d^*$, 
the dual space of $F_d$.
If $L_v$ depends continuously on $f$ but is not necessarily linear 
the class is denoted by $\Lambda^{\rm cont}$. 
Note that in both the cases also $N$ is continuous 
and we obviously have $\Lambda^{\rm all} \subset \Lambda^{\rm cont}$. 

Furthermore, we distinguish between \textit{adaptive} and \textit{non-adaptive} 
algorithms. The latter case is described above in formula \link{non-adapt}, 
where $L_v$ does not depend on the previously computed values 
$L_1(f),\ldots,L_{v-1}(f)$. 
In contrast, we also discuss algorithms of the form $A_{n,d}=\phi \circ N$ with
\begin{gather*}
        N(f)= \left( L_1(f), L_2(f;y_1), \ldots, L_n(f;y_1,\ldots,y_{n-1})\right), \qquad f\in F_d,
\end{gather*}
where $y_1=L_1(f)$ and $y_v=L_v(f;y_1,\ldots,y_{v-1})$ for $v=2,3,\ldots,n$. 
If $N$ is adaptive we restrict ourselves to the case where $L_v$ depends 
linearly on $f$, \ie
$L_v(\,\cdot\,; y_1,\ldots,y_{v-1}) \in \Lambda^{\rm all}$.

In all cases of information maps, the mapping $\phi$ can be chosen 
arbitrarily and is not necessarily linear or continuous. 
The smallest class of algorithms under consideration is the class 
of linear, non-adaptive algorithms of the form
\begin{gather*}
        A_{n,d}f=\sum_{v=1}^n L_v(f) \cdot g_v,
\end{gather*}
with some $g_v \in G_d$ and $L_v \in \Lambda^{\rm all}$.
We denote the class of all such algorithms by $\A_{n}^{\rm lin}$. 
On the other hand, the most general classes consist of algorithms 
$A_{n,d}=\phi \circ N$, where $\phi$ is arbitrary and $N$ either uses 
non-adaptive continuous or adaptive linear information. 
We denote the respective classes by $\A_n^{\rm cont}$ and $\A_n^{\rm adapt}$.

For the proof that the upper bound given in \autoref{theo_upperbound} is sharp 
we use a generalization of Lemma 1 in W.~\cite{W11}.

\begin{lemma}\label{needed_lemma}
				Suppose $S$ to be a homogeneous operator between linear normed spaces $X$ and $Y$,
				\ie$S(\alpha x)=\alpha S(x)$ for all $x\in X$ and $\alpha\in\R$. 
				Furthermore, assume that $V \subset X$ is a linear subspace with dimension $m$ and 
				there exists a constant $a \geq 0$ such that
				\begin{gather*}
								a \cdot \norm{f \sep X} \leq \norm{S(f) \sep Y} \quad \text{for all} \quad f\in V.
				\end{gather*}
				Then for every $n<m$ and every algorithm $A_n \in \A_n^{\rm cont} \cup \A_n^{\rm adapt}$
				\begin{gather*}
								e^{\rm wor}(A_n; S, X) = \sup_{f\in \B(X)} \norm{S(f) - A_n(f) \sep Y} \geq a.
				\end{gather*}
\end{lemma}

\begin{proof}
It is well-known that for $A_n=\phi \circ N$ with $n<m$ there exists $f^* \in V$ 
such that $N(f^*)=N(-f^*)$ and $\norm{f\sep X}=1$. Thus, $A_n(f^*)=A_n(-f^*)$.
For a more detailed view, see, W.~\cite[Lemma 1]{W11} and the references in there.
Using the triangle inequality for $Y$ we obtain
\begin{eqnarray*}
				e^{\rm wor}(A_n; X) 
				&\geq& \max{\norm{S(\pm f^*)-A_n(\pm f^*) \sep Y}} = \max{\norm{S(f^*) \pm A_n(f^*) \sep Y}} \\
				&\geq& \frac{1}{2} (\norm{ S(f^*) + A_n(f^*) \sep Y} + \norm{ S(f^*) - A_n(f^*) \sep Y}) \\
				&\geq& \frac{1}{2} \norm{2 \cdot S(f^*) \sep Y} \geq a \norm{f^* \sep X} = a
\end{eqnarray*}
and the proof is complete.
\end{proof}

Now let $X=P_I(H_d)$ and $Y=G_d$.
Furthermore, for a given $n\in\N_0$, 
define $a=\sqrt{\lambda_{d,\psi(n+1)}}$ and consider 
$V = \spann{\xi_{\psi(1)},\ldots, \xi_{\psi(n+1)}} \subset P_I(H_d)$.
Then, obviously, $\dim V = n+1=m>n$ .
With the representation~\link{Sdf} and formula \link{Sdxi_orth} 
from the proof of \autoref{theo_upperbound} we conclude
\begin{gather*}
				\norm{S_d f \sep G_d}^2 
				= \sum_{v=1}^{n+1} \distr{f}{\xi_{\psi(v)}}^2 \cdot \lambda_{d,\psi(v)} 
				\geq \lambda_{d,\psi(n+1)} \sum_{v=1}^{n+1} \distr{f}{\xi_{\psi(v)}}^2 
				= a^2 \norm{f \sep H_d}^2, \quad f\in V,
\end{gather*}
where we used the monotonicity of 
$\{ \lambda_{d,\psi(v)} \}_{v\in\N}$ and Parseval's identity.
This leads to the desired lower bound result:

\begin{prop}[Lower bound]
				Under the assumptions of \autoref{theo_opt_algo}
				the $n$-th minimal error with respect to the class 
				$\A_n^{\rm cont} \cup \A_n^{\rm adapt}$ is bounded from below by
				\begin{gather*}
								e(n,d; P_I(H_d)) 
								= \inf_{A_{n,d}} e^{\rm wor}(A_{n,d}; P_I(H_d)) 
								\geq \sqrt{\lambda_{d,\psi(n+1)}} 
								\quad \text{for all} \quad d\in\N, n\in\N_0.
				\end{gather*}
\end{prop}

Hence, this together with \autoref{theo_upperbound} shows that $A^*_{n,d}$ 
given in \link{opt_algo} is $n$-th optimal with respect to the class
$\A_n^{\rm cont} \cup \A_n^{\rm adapt}$ as claimed in \autoref{theo_opt_algo}.


\subsection*{Proof of \autoref{prop_NW} in \autoref{sect_prelim}}


\begin{proof}
If the problem is polynomially tractable then there exist constants $C,p>0$ and $q\geq0$ such that
for all $d\in \N$ and 
$\epsilon \in ( 0, 1]$
\begin{gather*}
				n(\epsilon,d) = n(\epsilon,d; F_d) \leq C \cdot \epsilon^{-p} \cdot d^q.
\end{gather*}
Here, $\epsilon_d^{\rm init} = e(0,d)=\sqrt{\lambda_{d,1}}>0$ denotes the initial error of $T_d$.
Since $e(n,d)=\sqrt{\lambda_{d,n+1}}$ it is $n(\epsilon,d)= \# \{i\in\N \sep \lambda_{d,i} > \epsilon^2\}$ 
and therefore $\lambda_{d, n(\epsilon,d)+1} \leq \epsilon^2$. The non-increasing ordering of $(\lambda_{d,i})_{i\in\N}$
implies
\begin{gather*}
				\lambda_{d, \floor{C\epsilon^{-p}d^q}+1} \leq \epsilon^2.
\end{gather*}
If we set $i=\floor{C\cdot \epsilon^{-p} \cdot d^q}+1$ and vary 
$\epsilon \in (0,1]$ 
then $i$ takes the values 
$\floor{C \cdot d^q}+1$, 
$\floor{C\cdot d^q}+2$, and so forth. 
On the other hand, we have $i \leq C\epsilon^{-p}d^q+1$, 
which is equivalent to $\epsilon^2 \leq (Cd^q/(i-1))^{2/p}$ if $i\geq 2$.
Thus,
\begin{gather*}
				\lambda_{d,i} \leq \lambda_{d, n(\epsilon,d)+1} \leq \epsilon^2 \leq \left( \frac{Cd^q}{i-1} \right)^{2/p} 
				\quad \text{for all} \quad i \geq \max{2,\floor{C\cdot d^q}+1}.
\end{gather*}
Choosing $\tau \geq 0$ and 
$f(d) = \ceil{(1+C)\cdot d^q} \geq \max{2,\floor{C\cdot d^q}+1}$
we conclude
\begin{gather*}
				\sum_{i = f(d)}^\infty \lambda_{d,i}^\tau 
				\leq \sum_{i = f(d)}^\infty \left( \frac{Cd^q}{i-1} \right)^{2\tau/p} 
				= (Cd^q)^{2\tau/p} \sum_{i = f(d)-1}^\infty \frac{1}{i^{2\tau/p}} 
				\leq (Cd^q)^{2\tau/p} \cdot \zeta\left(\frac{2\tau}{p}\right),
\end{gather*}
where $\zeta$ denotes the Riemann zeta function. 
In other words, if $\tau > p/2>0$ then
\begin{gather*}
				\frac{1}{d^{2q/p}} \left( \sum_{i = f(d)}^\infty \lambda_{d,i}^\tau \right)^{1/\tau} 
				\leq C^{2/p} \cdot \zeta\left(\frac{2\tau}{p}\right)^{1/\tau}<\infty
				\quad \text{for every} \quad d\in\N.
\end{gather*}
Setting $r = 2q/p$ proves the assertion, as well as the claimed bound on $C_\tau$. 

Conversely, assume now that \link{sup_condition} holds with
\begin{gather*}
				f(d)
				=\ceil{C\cdot \left(\min{\epsilon_d^{\rm init},1}\right)^{-p}\cdot d^{q}}
				\quad \text{where} \quad C>0 \quad \text{and} \quad p,q \geq 0.
\end{gather*}
That is, for some $r\geq 0$ and $\tau>0$ we have
\begin{gather*}
				0 < C_2 = \sup_{d\in\N} \frac{1}{d^r} \left( \sum_{i=f(d)}^\infty \lambda_{d,i}^\tau \right)^{1/\tau} < \infty.
\end{gather*}
For $n \geq f(d)$, the ordering of $(\lambda_{d,i})_{i\in\N}$ implies $\sum_{i=f(d)}^n \lambda_{d,i}^\tau \geq \lambda_{d,n}^\tau \cdot (n - f(d)+1)$. Hence,
\begin{gather*}
				\lambda_{d,n} \cdot (n - f(d)+1)^{1/\tau} \leq \left( \sum_{i=f(d)}^n \lambda_{d,i}^\tau \right)^{1/\tau} \leq \left( \sum_{i=f(d)}^\infty \lambda_{d,i}^\tau \right)^{1/\tau} \leq C_2 \, d^r,
\end{gather*}
or, respectively, $\lambda_{d,n+1} \leq C_2 \, d^r \cdot ((n+1) - f(d)+1)^{-1/\tau}$, for all $n\geq f(d)-1$. 
Note that for $\epsilon \in (0, \min{\epsilon_d^{\rm init},1}]$ we have 
$C_2 \, d^r \cdot ((n+1) - f(d)+1)^{-1/\tau} \leq \epsilon^2$
if and only if
\begin{gather*}
				n \geq n^* = \ceil{\left( \frac{C_2 \, d^r}{\epsilon^2} \right)^\tau }+ f(d)-2.
\end{gather*}
In particular, it is $\lambda_{d,n+1} \leq \epsilon^2$ at least for $n\geq \max{n^*,f(d)-1}$.
Therefore, for every $d\in\N$ and for all $\epsilon \in (0, \min{\epsilon_d^{\rm init},1}]$ it is
\begin{align*}
				n(\epsilon,d;F_d) 
				&\leq \max{n^*,f(d)-1} 
				\leq f(d)-1 + \left( \frac{C_2 \, d^r}{\epsilon^2} \right)^\tau \\
				&\leq C \cdot \left(\min{\epsilon_d^{\rm init},1}\right)^{-p}\cdot d^{q} + C_2^\tau\, \epsilon^{-2\tau} \, d^{r\tau}\\
				&\leq (C+C_2^\tau) \cdot \epsilon^{-\max{p,2\tau}} \cdot d^{\max{q, r\tau}}.
\end{align*} 
Hence, the problem is polynomially tractable.
\end{proof}


\subsection*{An explicit proof of \autoref{lemma_symNEW} in \autoref{sect_symprob}}


\begin{proof}
\textit{Step 1}. By induction on $s$ we first show for every fixed $m\in\N$
\begin{gather}\label{case_mNEW}
				\sum_{\substack{k\in\N^s,\\m\leq k_1\leq\ldots\leq k_s}} \mu_{s,k}
				= \mu_m^s + \sum_{l=1}^s \mu_m^{s-l}\sum_{ \substack{ j^{(l)}\in\N^l,\\m+1\leq j_1^{(l)}\leq\ldots\leq j_l^{(l)} } } \mu_{l,j^{(l)}}
				\quad \text{for all} \quad s\in\N.
\end{gather}
Easy calculus shows that this holds at least for the initial step $s=1$.
Therefore, assume~\link{case_mNEW} to be true for some $s\in\N$.
Then
\begin{align*}
				\sum_{\substack{k\in\N^{s+1},\\m\leq k_1\leq\ldots\leq k_{s+1}}} \mu_{s+1,k}
				&= \sum_{k_1=m}^\infty \mu_{k_1} \sum_{\substack{h\in\N^{s},\\k_1\leq h_1\leq\ldots\leq h_{s}}} \mu_{s,h}
				= \mu_m \sum_{\substack{h\in\N^{s},\\m\leq h_1\leq\ldots\leq h_{s}}} \mu_{s,h} + \sum_{k_1=m+1}^\infty \mu_{k_1} \sum_{\substack{h\in\N^{s},\\k_1\leq h_1\leq\ldots\leq h_{s}}} \mu_{s,h} \\
				&= \mu_m \sum_{\substack{h\in\N^{s},\\m\leq h_1\leq\ldots\leq h_{s}}} \mu_{s,h} + \sum_{\substack{k\in\N^{s+1},\\m+1\leq k_1\leq\ldots\leq k_{s+1}}} \mu_{s+1,k}
\end{align*}
Now, by inserting the induction hypothesis for the first sum and 
renaming $k$ to $j^{(s+1)}$ in the remaining sum, we conclude
\begin{gather*}
				\sum_{\substack{k\in\N^{s+1},\\m\leq k_1\leq\ldots\leq k_{s+1}}} \mu_{s+1,k}
				= \mu_m^{s+1} + \sum_{l=1}^s \mu_m^{s+1-l}\sum_{ \substack{ j^{(l)}\in\N^l,\\m+1\leq j_1^{(l)}\leq\ldots\leq j_l^{(l)} } } \mu_{l,j^{(l)}} + \sum_{\substack{j^{(s+1)}\in\N^{s+1},\\m+1\leq j^{(s+1)}_1\leq\ldots\leq j^{(s+1)}_{s+1}}} \mu_{s+1,j^{(s+1)}}.
\end{gather*}
Hence, \link{case_mNEW} also holds for $s+1$ and the induction is complete.

\textit{Step 2}. Here we prove \link{estimate_VNEW} via another induction on $V\in\N_0$.
Therefore, let $d\in\N$ be fixed arbitrarily.
The initial step, $V=0$, corresponds to \link{case_mNEW} for $s=d$ and $m=1$.
Thus, assume \link{estimate_VNEW} to be true for some fixed $V\in\N_0$.
Then it is
\begin{align*}
					\sum_{\substack{k\in\N^d,\\1\leq k_1\leq\ldots\leq k_d}} \mu_{d,k}
					&\leq \mu_1^d \, d^V \left( 1 + V + \sum_{L=1}^d  \mu_1^{-L} \sum_{ \substack{ j^{(L)}\in\N^L,\\V+2\leq j_1^{(L)}\leq\ldots\leq j_L^{(L)} } } \mu_{L,j^{(L)}} \right) \\
					&= \mu_1^d \, d^V \left( 1 + V + \sum_{L=1}^d \mu_1^{-L} \left( \mu_{V+2}^L + \sum_{l=1}^L \mu_{V+2}^{L-l} \sum_{ \substack{ j^{(l)}\in\N^l,\\(V+2)+1\leq j_1^{(l)}\leq\ldots\leq j_l^{(l)} } } \mu_{l,j^{(l)}} \right) \right),
\end{align*}
using \link{case_mNEW} for $s=L$ and $m=V+2$.
Now we estimate $1+V$ by $d(1+V)$, take advantage of the non-increasing ordering of $(\mu_m)_{m\in\N}$ and extend the inner sum from $L$ to $d$
in order to obtain
\begin{align*}
					\sum_{\substack{k\in\N^d,\\1\leq k_1\leq\ldots\leq k_d}} \mu_{d,k}
					\leq \mu_1^{d} \, d^{V+1} \left( 1 + (V + 1) + \sum_{l=1}^d \mu_1^{-l} \sum_{ \substack{ j^{(l)}\in\N^l,\\(V+1)+2\leq j_1^{(l)}\leq\ldots\leq j_l^{(l)} } } \mu_{l,j^{(l)}} \right).
\end{align*}
Since this estimate corresponds to \link{estimate_VNEW} for $V+1$ the claim is proven.
\end{proof}


\subsection*{Proof of \autoref{theo_exact} in \autoref{sect_asy_abs}}


\begin{proof}
Note that due to $\lim_{m\nach \infty} \lambda_m=0$ the quantity $i_d$ is well defined, because $(\alpha_i)_{i\in\N}$ 
is a non-increasing sequence which tends to zero for $i$ tending to infinity, and $i_d(\delta^2)>1$
for $\delta < \epsilon_d^{\rm init}$.
Furthermore, we have 
\begin{gather*}
			i_1(\delta^2) = \#\{m\in\N \sep \lambda_m > \delta^2\} + 1 = n^{\rm ent}(\delta,1) + 1 =  n^{\rm asy}(\delta,1) +1
\end{gather*}
and if $d\geq 2$ we can rewrite $i_d$ to obtain
\begin{gather*}
				i_d(\delta^2) = \min{i \in \N \sep \lambda_{i+1} \cdot \ldots \cdot \lambda_{i+d-1} \leq \frac{1}{\lambda_i} \delta^2}.
\end{gather*}
Hence, for every $k=(k_1,\ldots,k_{d-1}) \in \nabla_{d-1}$ with $k_1 > i_d(\delta^2)$ it is
\begin{gather*}
				\lambda_{d-1,k}=\lambda_{k_1}\cdot\ldots\cdot\lambda_{k_{d-1}} \leq \lambda_{i_d(\delta^2)+1}\cdot\ldots\cdot\lambda_{i_d(\delta^2)+d-1} \leq \frac{1}{\lambda_{i_d(\delta^2)}} \delta^2
\end{gather*}
or, equivalently, 
\begin{gather*}
				\left\{k \in \nabla_{d-1} \sep i < k_1 \text{ and } \lambda_{d-1,k} > \frac{1}{\lambda_i} \delta^2 \right\} = \leer \quad \text{for all} \quad i \geq i_d(\delta^2).
\end{gather*}
This leads to the disjoint decomposition of
\begin{eqnarray*}
				\{j \in \nabla_d \sep \lambda_{d,j} > \delta^2\} 
				&=& \left\{ j=(i,k) \in \N\times\nabla_{d-1} \sep i < k_1 \text{ and } \lambda_{d-1,k} > \frac{1}{\lambda_i} \delta^2 \right\} \\
				&=& \bigcup_{i=1}^{i_d(\delta^2)-1} \left\{ (i,k) \sep k\in\nabla_{d-1} \text{ such that } i<k_1 \text{ and } \lambda_{d-1,k} > \frac{1}{\lambda_{i}} \delta^2 \right\}.
\end{eqnarray*}
Therefore, the information complexity of the $d$-variate problem is given by
\begin{eqnarray*}
				n^{\rm asy}(\epsilon,d) 
				&=& \# \{ j \in \nabla_d \sep \lambda_{d,j}>\epsilon^2\}
				= \sum_{i=1}^{i_d(\epsilon^2)-1} \# \left\{ k \in \nabla_{d-1} \sep i<k_1 \text{ and } \lambda_{d-1,k} > \frac{1}{\lambda_i} \epsilon^2 \right\} \\
				&=& \sum_{l_1=2}^{i_d(\epsilon^2)} \# \left\{ k \in \nabla_{d-1} \sep l_1 \leq k_1 \text{ and } \lambda_{d-1,k} > \frac{1}{\lambda_{l_1-1}} \epsilon^2 \right\}.
\end{eqnarray*}
Obviously, for fixed $l_1 \in \{2,\ldots,i_d(\epsilon^2)\}$, we can repeat this procedure and obtain
\begin{align*}
				&\# \{ j \in \nabla_{d-1} \sep l_1 \leq j_1 \text{ and } \lambda_{d-1,j}>\delta^2\} \\
				&\qquad\qquad = \sum_{l_2=l_1+1}^{i_{d-1}(\delta^2)} \# \left\{ k \in \nabla_{d-2} \sep l_2 \leq k_1 \text{ and } \lambda_{d-2,k} > \frac{1}{\lambda_{l_2-1}} \delta^2 \right\},
\end{align*}
if $d>2$ and $\delta^2 = \epsilon^2 / \lambda_{l_1-1}$. Note that $\epsilon < \epsilon_d^{\rm init}$ 
implies $i_{d-1}(\delta^2)\geq l_1+1$ such that  $\{l_1+1,\ldots,i_{d-1}(\delta)^2\} \neq \leer$.
Iterating the argument we get
\begin{align*}
				&n^{\rm asy}(\epsilon,d) \\
				&\qquad = \sum_{l_1=2}^{i_d(\epsilon^2)} \sum_{l_2=l_1+1}^{i_{d-1}(\epsilon^2 / \lambda_{l_1-1})} \ldots \sum_{l_{d-1}=l_{d-2}+1}^{i_{2}(\epsilon^2 / [\lambda_{l_1-1}\cdot\ldots\cdot\lambda_{l_{d-2}-1}])} 
				\# \left\{ k \in \nabla_{1} \sep l_{d-1} \leq k_1 \text{ and } \lambda_{1,k} > \frac{1}{\lambda_{l_{d-1}-1}}\delta^2 \right\}
\end{align*}
with $\delta^2 = \epsilon^2 / [\lambda_{l_1-1} \cdot\ldots\cdot\lambda_{l_{d-2}-1}]$. 
It remains to calculate the cardinality of the last set.
Of course, we have
\begin{align*}
			&\left\{ k \in \nabla_{1} \sep l_{d-1} \leq k_1 \text{ and } \lambda_{1,k} > \frac{1}{\lambda_{l_{d-1}-1}} \delta^2 \right\} \\
			&\qquad\qquad= \left\{ k \in \N \sep l_{d-1} \leq k \text{ and } \lambda_{k} > \frac{1}{\lambda_{l_{d-1}-1}} \delta^2 \right\}\\
			&\qquad\qquad= \left\{ k \in \N \sep \lambda_{k} > \frac{1}{\lambda_{l_{d-1}-1}} \delta^2 \right\} \setminus \left\{ k \in \{1,\ldots,l_{d-1}-1\} \sep \lambda_{k} > \frac{1}{\lambda_{l_{d-1}-1}} \delta^2 \right\}.
\end{align*}
The first of these sets in the last line contains exactly 
$n^{\rm ent}(\delta/\sqrt{\lambda_{l_{d-1}-1}}, 1)$ elements.
On the other hand, if $k\leq l_{d-1} \leq i_2(\delta^2)$ then 
\begin{gather*}
				\lambda_k\lambda_{l_{d-1}-1}\geq \lambda_{i_2(\delta^2)} \lambda_{i_2(\delta^2)-1} > \delta^2,
\end{gather*}
where the last inequality holds due to the definition of $i_2(\delta^2)$. 
Therefore, the last set coincides with $\{1,\ldots,l_{d-1}-1\}$ 
and its cardinality is equal to $l_{d-1}-1$.
Furthermore, note that the estimate also shows that 
$n^{\rm ent}(\delta/\sqrt{\lambda_{l_{d-1}-1}}, 1)$
is at least equal to $l_{d-1}$.
Thus,
\begin{gather*}
				\# \left\{ k \in \nabla_{1} \sep l_{d-1} \leq k_1 \text{ and } \lambda_{1,k} > \frac{1}{\lambda_{l_{d-1}-1}} \delta^2 \right\} 
				= n^{\rm ent}\left( \delta/\sqrt{\lambda_{l_{d-1}-1}}, 1\right) - l_{d-1} + 1 \geq 1
\end{gather*}
and the proof is complete.
\end{proof}


\subsection*{Proof of \autoref{prop_NW2} in \autoref{sect_normkrit}}


One possibility to prove the second point of \autoref{prop_NW2} 
is to apply \autoref{prop_NW} to a scaled problem 
$\{\tilde{T}_d\}$ such that $\tilde{W_d}={ \tilde{T}_d }^{\dagger} \tilde{T}_d$ 
possesses the eigenvalues $\tilde{\lambda}_{d,i} = \lambda_{d,i}/\lambda_{d,1}$ 
for $i\in\N$. 
Then the initial error of $\tilde{T}_d$ equals $1$ such that 
$f$ in \autoref{prop_NW} does not depend on~$p$.
That is, we can choose $f(d)=\ceil{C\, d^{q}}+1$ for some 
$q\geq 0$ in both the assertions.
In order to see why we even can take $f(d) \equiv 1$ in the first point 
and for the sake of completeness we also add a direct proof 
for this proposition.

\begin{proof}
If $\{T_d\}$ is polynomially tractable with respect to the normalized error criterion
then there exist constants $C,p>0$ and $q\geq0$ such that
for all $d\in \N$ and $\epsilon' \in ( 0, 1 ]$
\begin{gather*}
				n(\epsilon' \cdot \epsilon_d^{\rm init},d) = 
				n(\epsilon' \cdot \epsilon_d^{\rm init},d; F_d) 
				\leq C \cdot \left(\epsilon'\right)^{-p} \cdot d^q.
\end{gather*}
As before the quantity $\epsilon_d^{\rm init} = \sqrt{\lambda_{d,1}}>0$ denotes the initial error of $T_d$ and 
$\epsilon'$ is the (multiplicative) improvement of it.
Since $e(n,d)=\sqrt{\lambda_{d,n+1}}$ it is $n(\epsilon,d)= \# \{i\in\N \sep \lambda_{d,i} > \epsilon^2\}$
where $\epsilon = \epsilon' \cdot \epsilon_d^{\rm init}$. 
Therefore, $\lambda_{d, n(\epsilon' \cdot \epsilon_d^{\rm init},d)+1} \leq \left(\epsilon'\right)^2 \cdot \lambda_{d,1} $. 
Hence, the non-increasing ordering of $(\lambda_{d,i})_{i\in\N}$
implies in this setting
\begin{gather*}
				\lambda_{d, \floor{C \left(\epsilon'\right)^{-p} d^{q}}+1} \leq \left(\epsilon'\right)^2 \cdot \lambda_{d,1}.
\end{gather*}
If we set $i=\floor{C \left(\epsilon'\right)^{-p} d^{q}}+1$ and vary 
$\epsilon' \in (0,1]$ then $i$ takes the
values $\floor{C d^q}+1$, $\floor{C d^q}+2$ and so on.
Again we have $1 \leq i \leq C \left(\epsilon'\right)^{-p}d^q+1$ 
on the other hand, which is equivalent to 
$\left(\epsilon'\right)^2 \leq (Cd^q/(i-1))^{2/p}$ if $i\geq2$.
Thus,
\begin{gather*}
				\lambda_{d,i} 
				\leq \lambda_{d, n(\epsilon' \cdot \epsilon_d^{\rm init},d)+1} 
				\leq \left(\epsilon'\right)^2 \cdot \lambda_{d,1} 
				\leq \left( \frac{Cd^q}{i-1} \right)^{2/p} \cdot \lambda_{d,1} \quad \text{for all} \quad i \geq \max{2, \floor{C d^q}+1}.
\end{gather*}
Choosing $\tau \geq 0$ and $f^{*}(d)= \ceil{(1+C)\, d^q} \geq \max{2, \floor{C d^q}+1}$ we conclude here
\begin{gather*}
				\sum_{i = f^{*}(d)}^\infty \left(\frac{\lambda_{d,i}}{\lambda_{d,1}}\right)^\tau 
				\leq \sum_{i = f^{*}(d)}^\infty \left( \frac{Cd^q}{i-1} \right)^{2\tau/p} 
				\leq (Cd^q)^{2\tau/p} \cdot \zeta\left(\frac{2\tau}{p}\right),
\end{gather*}
where $\zeta$ again is the Riemann zeta function. 
On the other hand, it is obvious that
\begin{gather*}
				\sum_{i = 1}^{f^{*}(d)-1} \left(\frac{\lambda_{d,i}}{\lambda_{d,1}}\right)^\tau \leq f^*(d)-1 \leq (1+C) \, d^{q\cdot 2\tau/p},
\end{gather*}
because $\lambda_{d,i}\leq \lambda_{d,1}$ for all $i\in\N$.
Therefore, if $\tau > p/2$, 
\begin{gather*}
				\frac{1}{d^{2\tau q/p}} \sum_{i = 1}^\infty \left(\frac{\lambda_{d,i}}{\lambda_{d,1}}\right)^\tau
				\leq 1+C + C^{2\tau/p} \cdot \zeta\left(\frac{2\tau}{p}\right)<\infty
\end{gather*}
for all $d\in\N$. This proves the assertion setting $r \geq 2q/p$. 

The proof of the second point again works like for \autoref{prop_NW}.
Assume that \link{sup_condition2} holds with 
$f(d) = \ceil{C \, d^{q}}$, where $C>0$ and $q \geq 0$.
That is, for some $r\geq 0$ and $\tau>0$ we have
\begin{gather*}
				C_\tau = \sup_{d\in\N} \frac{1}{d^r} \left( \sum_{i=f(d)}^\infty \left(\frac{\lambda_{d,i}}{\lambda_{d,1}}\right)^\tau \right)^{1/\tau} < \infty.
\end{gather*}
Since $(\lambda_{d,i})_{i\in\N}$ is assumed to be non-increasing 
the same also holds for the rescaled sequence 
$(\lambda_{d,i} / \lambda_{d,1})_{i\in\N}$ 
such that 
$\sum_{i=f(d)}^n (\lambda_{d,i}/\lambda_{d,1})^\tau \geq (\lambda_{d,n}/\lambda_{d,1})^\tau \cdot (n - f(d)+1)$ 
for $n \geq f(d)$.
Hence,
\begin{gather*}
				\frac{\lambda_{d,n}}{\lambda_{d,1}} \cdot (n - f(d)+1)^{1/\tau} 
				\leq \left( \sum_{i=f(d)}^\infty \left(\frac{\lambda_{d,i}}{\lambda_{d,1}}\right)^\tau \right)^{1/\tau} 
				\leq C_\tau \, d^r,
\end{gather*}
or, respectively, 
$\lambda_{d,n+1} \leq C_\tau \, d^r \cdot ((n+1) - f(d)+1)^{-1/\tau} \cdot \lambda_{d,1}$ 
for all $n\geq f(d)-1$. 
As before we have 
$C_\tau \, d^r \cdot ((n+1) - f(d)+1)^{-1/\tau} \leq \left(\epsilon'\right)^2$, for $\epsilon' \in (0, 1]$,
if and only if
\begin{gather*}
				n \geq n^{*} = \ceil{\left( \frac{C_\tau \, d^r}{\left(\epsilon'\right)^2} \right)^\tau }+ f(d)-2.
\end{gather*}
In particular, $\lambda_{d,n+1} \leq \left(\epsilon'\right)^2\cdot \lambda_{d,1}$ 
at least for $n\geq \max{n^{*},f(d)-1}$.
Therefore, we conclude in this setting for all $\epsilon' \in (0, 1]$ and every $d\in\N$
\begin{align*}
				n(\epsilon' \cdot \epsilon_d^{\rm init},d; F_d) 
				&\leq \max{n^{*},f(d)-1} 
				\leq f(d)-1 + \left( \frac{C_\tau \, d^r}{\left( \epsilon'\right)^2} \right)^\tau 
				\leq C\, d^q + C_\tau^\tau \left( \epsilon' \right)^{-2\tau} d^{r\tau} \\
				& \leq (C+C_\tau^\tau) \cdot \left( \epsilon'\right)^{-2\tau} \cdot d^{\max{q, r\tau}}.
\end{align*}
Hence, the problem is polynomially tractable.
Furthermore, strong polynomial tractability 
holds if $r=q=0$.
\end{proof}

\end{document}